\documentclass[11pt]{article}

\usepackage{enumerate, amsmath, amsthm, amsfonts, amssymb, xy, color,
  mathrsfs, graphicx, paralist}
\usepackage[margin=1in]{geometry} 
\usepackage[colorlinks=true, urlcolor=blue, linkcolor=blue, citecolor=blue,  bookmarksdepth=2]{hyperref}

\usepackage{tocloft}
\usepackage{tabmac}

\allowdisplaybreaks

\input xy
\xyoption{all}

\numberwithin{equation}{section}
\newtheorem{theorem}[equation]{Theorem}

\newtheorem{proposition}[equation]{Proposition}
\newtheorem{lemma}[equation]{Lemma}

\theoremstyle{definition}
\newtheorem{rmk}[equation]{Remark}
\newenvironment{remark}[1][]{\begin{rmk}[#1] \pushQED{\qed}}{\popQED \end{rmk}}
\newtheorem{eg}[equation]{Example}
\newenvironment{example}[1][]{\begin{eg}[#1] \pushQED{\qed}}{\popQED \end{eg}}
\newtheorem{defn}[equation]{Definition}
\newenvironment{definition}[1][]{\begin{defn}[#1]\pushQED{\qed}}{\popQED \end{defn}}
\newtheorem{construct}[equation]{Construction}
\newenvironment{construction}[1][]{\begin{construct}[#1]\pushQED{\qed}}{\popQED \end{construct}}

\newcommand{\bC}{\mathbf{C}}
\newcommand{\cC}{\mathcal{C}}

\newcommand{\cE}{\mathcal{E}}

\newcommand{\bF}{\mathbf{F}}

\newcommand{\cH}{\mathcal{H}}

\newcommand{\rH}{\mathrm{H}}

\newcommand{\cI}{\mathcal{I}}

\newcommand{\bL}{\mathbf{L}}
\newcommand{\cL}{\mathcal{L}}

\newcommand{\rL}{\mathrm{L}}

\newcommand{\cM}{\mathcal{M}}

\newcommand{\cO}{\mathcal{O}}

\newcommand{\bP}{\mathbf{P}}

\newcommand{\cQ}{\mathcal{Q}}

\newcommand{\cR}{\mathcal{R}}

\newcommand{\rR}{\mathrm{R}}

\newcommand{\bS}{\mathbf{S}}
\newcommand{\cS}{\mathcal{S}}

\newcommand{\cT}{\mathcal{T}}

\newcommand{\cU}{\mathcal{U}}

\newcommand{\cX}{\mathcal{X}}

\newcommand{\cY}{\mathcal{Y}}

\newcommand{\bZ}{\mathbf{Z}}
\newcommand{\cZ}{\mathcal{Z}}

\newcommand{\rd}{\mathrm{d}}

\newcommand{\fg}{\mathfrak{g}}

\newcommand{\fh}{\mathfrak{h}}
\newcommand{\rh}{\mathrm{h}}

\newcommand{\sk}{\mathscr{k}}

\newcommand{\fl}{\mathfrak{l}}

\newcommand{\fm}{\mathfrak{m}}

\newcommand{\fn}{\mathfrak{n}}

\newcommand{\fp}{\mathfrak{p}}

\newcommand{\fs}{\mathfrak{s}}

\newcommand{\ft}{\mathfrak{t}}

\newcommand{\fz}{\mathfrak{z}}

\renewcommand{\phi}{\varphi}
\renewcommand{\emptyset}{\varnothing}
\newcommand{\eps}{\varepsilon}

\renewcommand{\tilde}[1]{\widetilde{#1}}
\newcommand{\ol}[1]{\overline{#1}}
\newcommand{\ul}[1]{\underline{#1}}

\newcommand{\arxiv}[1]{\href{http://arxiv.org/abs/#1}{{\tt arXiv:#1}}}

\makeatletter
\def\Ddots{\mathinner{\mkern1mu\raise\p@
\vbox{\kern7\p@\hbox{.}}\mkern2mu
\raise4\p@\hbox{.}\mkern2mu\raise7\p@\hbox{.}\mkern1mu}}
\makeatother

\DeclareMathOperator{\codim}{codim}

\renewcommand{\hom}{\operatorname{Hom}}

\DeclareMathOperator{\ann}{Ann}

\DeclareMathOperator{\rank}{rank}
\DeclareMathOperator{\ext}{Ext}

\DeclareMathOperator{\Pf}{Pf}
\DeclareMathOperator{\Sym}{Sym}

\DeclareMathOperator{\depth}{depth}

\DeclareMathOperator{\pdim}{pdim}

\DeclareMathOperator{\Spec}{Spec}
\DeclareMathOperator{\grade}{grade}

\DeclareMathOperator{\sgn}{sgn}
\DeclareMathOperator{\Jac}{Jac}

\newcommand{\GL}{\mathbf{GL}}
\newcommand{\SL}{\mathbf{SL}}
\newcommand{\Sp}{\mathbf{Sp}}
\newcommand{\SO}{\mathbf{SO}}
\newcommand{\Gr}{\mathbf{Gr}}
\newcommand{\Fl}{\mathbf{Fl}}
\newcommand{\spin}{\mathbf{spin}}
\newcommand{\Spin}{\mathbf{Spin}}

\title{Moduli of Abelian varieties, Vinberg $\theta$-groups, and free resolutions}
\date{January 21, 2013}
\author{Laurent Gruson \and Steven V Sam \and Jerzy Weyman}
\begin{document}


\maketitle

\begin{center}
Dedicated to David Eisenbud on the occasion of his 65th birthday.
\end{center}

\begin{abstract}
We present a systematic approach to studying the geometric aspects of Vinberg $\theta$-representations. The main idea is to use the Borel--Weil construction for representations of reductive groups as sections of homogeneous bundles on homogeneous spaces, and then to study degeneracy loci of these vector bundles. Our main technical tool is to use free resolutions as an ``enhanced'' version of degeneracy loci formulas. We illustrate our approach on several examples and show how they are connected to moduli spaces of Abelian varieties. To make the article accessible to both algebraists and geometers, we also include background material on free resolutions and representation theory.
\end{abstract}

\tableofcontents

\section*{Introduction.}

The contents of these notes sit at the crossroads of representation theory, algebraic geometry, and commutative algebra, so we will explain each of these perspectives on our work before getting into the details.

From representation theory, we are considering the problem of classifying orbits in Vinberg $\theta$-representations $(G,U)$ \cite{vinberg}. From the point of view of (geometric) invariant theory these are the representations that are the simplest. One naturally gets a representation from a $\bZ$-grading of a Kac--Moody algebra $\fg$ and so one naturally gets a trichotomy of these representations according to the structure of $\fg$: finite type, affine type, and wild type. The $\theta$-representations come from finite and affine type. The $\theta$-representations of finite type have finitely many orbits and include many cases of classical interest, such as determinantal varieties. The study of the geometry and algebra of these orbits is undertaken in work in progress by Kra\'skiewicz--Weyman (starting with \cite{kw}). In the affine case, there are typically infinitely many orbits, but the $\theta$-representation $(G,U)$ has the property that the ring of semi-invariants $\Sym(U^*)^{(G,G)}$ is a polynomial ring and its unstable locus (nullcone) has finitely many orbits. This class of representations includes the adjoint representations of semisimple Lie algebras and share many features in common with them. 

From the point of view of algebraic geometry, we are giving geometric constructions of (torsors of) Abelian varieties from the data of a $G$-orbit in $U$. The GIT quotient $U/\!\!/G$ is isomorphic to a quotient space $\fh/W$ where $W \subset \GL(\fh)$ is a complex reflection group. In most of the cases that we consider, the quotients $\fh/W$ were considered previously and contain an open subset isomorphic to a moduli space of curves with a special kind of marked data. However, the constructions for Abelian varieties that we give seem to be new. One of the nice features of our work is that it is possible to make a lot of it explicit and turn it into input for the computer algebra system Macaulay2. For example, in the case of $\bigwedge^3 V$ ($V$ is a vector space of dimension 9), we give in \S\ref{section:w39m2code} a detailed explanation of how to calculate the ideals of the degeneracy loci under study (which include $(3,3)$-polarized Abelian surface torsors) starting with a $\GL(V)$-orbit.

From commutative algebra, this circle of ideas illustrates the power of a systematic use of perfect resolutions. As we will soon explain, the main tool that we employ is the Eagon--Northcott generic perfection theorem: using the minimal free resolutions of Kra\'skiewicz--Weyman (many of which are perfect resolutions), we can construct certain global sheafy complexes which specialize to locally free resolutions of some varieties of interest, such as the Abelian variety torsors mentioned above.

In particular, for $\theta$-representations of affine type, the GIT quotient $U /\!\!/ G$ is a weighted projective space, and hence rational, and it is of interest to know if the orbit space $U/G$ has a modular interpretation. The problem of studying the nilpotent orbits requires a different kind of approach and will not be discussed in these notes. For some of the examples of $\theta$-representations, it is easy to find such interpretations using constructions from standard linear algebra, such as determinants. However, many of them do not seem to have any obvious constructions associated with them. We take up a systematic approach to dealing with these representations which we now outline (see Construction~\ref{construct:main} for details). The main idea is to use information of the orbits in representations of finite type to bootstrap to the affine type case.

~

\begin{compactenum}[1.]
\item Using the Borel--Weil construction, one may realize $U$ as the space of sections of a homogeneous bundle $\cU$ on a homogeneous space $G/P$. This can usually be done in many different ways. For any point $x \in G/P$, the stabilizer of $x$ is a subgroup in $G$ which is conjugate to $P$, and for our choices of $P$, the action of $P$ on the fiber $\cU(x)$ will have finitely many orbits (and one can restrict the action to a certain reductive group $G' \subset P$ without affecting the orbit structure). 

\item These orbits can be glued together to get ``global orbit closures'' in the total space of $\cU$. Any vector $v \in U$ is then a section of $\cU$ and each of these global orbit closures gives ``degeneracy loci'' in $G/P$ by considering when $v(G/P)$ intersects a given global orbit closure. Hence the first step to understanding the geometry of these degeneracy loci should be to understand the ``local'' geometry of orbit closures in representations with finitely many orbits.

\item Vinberg's theory allows one to completely classify the orbits in these finite type representations. In many cases, one can calculate the minimal free resolutions of the coordinate rings of these orbit closures and various equivariant modules which are of interest. This program is taken up in the work of Kra\'skiewicz--Weyman.

\item For generic sections, the free resolutions can be turned into locally free resolutions for the degeneracy loci in question via the Eagon--Northcott generic perfection theorem. The power in this approach is that it allows one to obtain cohomological information about the varieties. This can sometimes be enough to determine the structure of the variety, such as showing it is an Abelian variety (which happens in many cases that we consider).
\end{compactenum}

~

In the case of a representation with finitely many orbits, there is no interesting moduli over an algebraically closed field, but it is often interesting to consider $\bZ$-forms of the group and its representation and to study the arithmetic orbits. This has been done in a series of papers by Bhargava \cite{bhargava}. Under that perspective, a later goal would be to understand the arithmetic orbits for the representations that we consider. Some examples of representations which have positive-dimensional quotients have been worked out by Ho in her thesis \cite{weiho}.

For the contents of the paper: \S\S\ref{section:freeres}--\ref{section:reptheory} are introductory in nature and provide background on free resolutions and the requisite representation theory that will be used in the notes. \S\ref{section:curves} contains some geometric preliminaries on Abelian varieties and moduli spaces of vector bundles on curves, as well as a more precise description of the steps outlined above. The rest of the notes are devoted to studying examples of $\theta$-representations of affine type. 

\subsection*{Notation and conventions.}

If $R$ is a graded ring, then $R(d)$ is the $R$-module $R$ with a
grading shift: $R(d)_i = R_{d+i}$. We define a {\bf graded local ring}
to be a positively graded ring $R$ whose degree 0 part is a field. In
this case, we denote $\fm = \bigoplus_{i > 0} R_i$. Given a free
$K$-module $E$, we think of $\Sym(E)$ as a graded ring by $\Sym(E)_i =
S^i E$. Furthermore, when we talk about modules over graded rings, we
will implicitly assume that they are also graded.

If $X$ is a scheme over a field $K$ and $E$ is a vector space, we let
$\ul{E}$ denote the trivial vector bundle $E \otimes \cO_X$. Also,
given a vector bundle $\cE$, we let $\det \cE$ denote its top exterior
power. 

For the sections involving examples, we will work over an algebraically closed field of characteristic 0, which we will just denote by the complex numbers $\bC$. Ultimately, the goal is to relax this assumption to other characteristics, or non-algebraically closed fields, so the introductory sections are written in this more general context. We will also make some comments throughout about this point. 

\subsection*{Acknowledgements.}

We thank Manjul Bhargava, Igor Dolgachev, Wei Ho, Steven Kleiman, Bjorn Poonen, and Jack Thorne for helpful discussions and Damiano Testa for explaining the proof of Theorem~\ref{thm:cohomologyabelian}. We also thank Igor Dolgachev for pointing out numerous references related to this work. Finally, we thank an anonymous referee for making some suggestions.

We also thank Federico Galetto and Witold Kra\'skiewicz for assistance with some computer calculations. The software LiE \cite{lie} and Macaulay2 \cite{m2} were helpful in our work.

Steven Sam was supported by an NDSEG fellowship while this work was
done. Jerzy Weyman was partially supported by NSF grant DMS-0901185.

\section{Free resolutions.} \label{section:freeres}

\subsection{Basic definitions.}

A lot of the foundational results in this section can be found in
\cite[Appendix]{brunsvetter}.

\begin{definition}
  Let $R$ be a commutative ring and $M$ be a finitely generated
  $R$-module. A complex of $R$-modules
  \[
  \bF_\bullet: \cdots \to \bF_i \xrightarrow{d_i} \bF_{i-1} \to \cdots
  \to \bF_0
  \]
  is a {\bf projective resolution} of $M$ if
  \begin{compactitem}
  \item each $\bF_i$ is a finitely generated projective $R$-module,
  \item $\rH_i(\bF_\bullet) = 0$ for $i>0$ and $\rH_0(\bF_\bullet) =
    M$.
  \end{compactitem}
The {\bf projective dimension} of $M$ (denoted $\pdim M$) is the minimum length of any projective resolution of $M$. {\bf Free resolutions} are projective resolutions where the $\bF_i$ are free modules. If $R$ is a (graded) local ring with maximal ideal $\fm$, then $\bF_\bullet$ is {\bf minimal} if
  \begin{compactitem}
  \item $d_i(\bF_i) \subseteq \fm \bF_{i-1}$ for all $i>0$.
  \end{compactitem}
  In the graded case, we will shift the gradings of the $\bF_i$ to
  assume that the differentials are homogeneous of degree 0.
\end{definition}

\begin{definition}
  Let $R$ be a Noetherian ring and let $M$ be a finitely generated
  $R$-module. A sequence $(r_1, \dots, r_n)$ of elements in $R$ is a
  {\bf regular sequence} on $M$ if
  \begin{compactitem}
  \item $r_1$ is not a zerodivisor or unit on $M$, and
  \item $r_i$ is not a zerodivisor or unit on $M/(r_1, \dots,
    r_{i-1})M$ for all $i>1$.
  \end{compactitem}
  For an ideal $I \subset R$, the {\bf depth} of $M$ (with respect to $I$) is the length of the longest regular sequence for $M$ which is
  contained in $I$. It is denoted by $\depth_I M$. If $R$ is local
  with maximal ideal $\fm$, we denote $\depth M = \depth_\fm M$. For
  an ideal $I \subset R$, the {\bf grade} of $I$ is the length of the
  longest regular sequence in $I$ for $R$. $M$ is {\bf perfect of
    grade $g$} if $g = \pdim M = \grade \ann M$. (In general, one has
  $\pdim M \ge \grade \ann M$.)

  Over a local Noetherian ring $R$, a finitely generated module $M$ is
  {\bf Cohen--Macaulay} if it is 0 or $\depth M = \dim M :=
  \dim(R/\ann(M))$. For a general Noetherian ring $R$, $M$ is
  Cohen--Macaulay if the localization $M_\fp$ is Cohen--Macaulay over
  $(R_\fp, \fp)$ for all prime ideals $\fp$ of $R$. A Noetherian ring
  is Cohen--Macaulay if it is a Cohen--Macaulay module over itself.
\end{definition}

\begin{theorem} Let $R$ be a Noetherian Cohen--Macaulay ring. 
  \begin{compactenum}[\rm 1.]
  \item For every ideal $I \subset R$, we have $\grade I = \codim I =
    \dim R - \dim(R/I)$.
  \item The polynomial ring $R[x]$ is Cohen--Macaulay.
  \item If an $R$-module $M$ is perfect, then it is Cohen--Macaulay. 
  \end{compactenum}
\end{theorem}

(The distinction between perfect and Cohen--Macaulay for a module over
a Cohen--Macaulay ring is the property of having finite projective
dimension.)

\begin{theorem}[Auslander--Buchsbaum formula] Suppose $R$ is a
  Noetherian (graded) local ring and that $M$ is a finitely generated
  $R$-module with $\pdim M < \infty$. Then
  \[
  \depth M + \pdim M = \depth R. 
  \]
\end{theorem}

\begin{theorem} Let $R$ be a Noetherian (graded) local ring and $M$ be
  a perfect $R$-module of grade $g$ with minimal free resolution
  $\bF_\bullet$. Then $\hom(\bF_\bullet, R)$ is a minimal free
  resolution of the perfect module $M^\vee = \ext^g_R(M,R)$, and
  $(M^\vee)^\vee \cong M$.
\end{theorem}

\begin{definition}
  If $M = R/I$, for an ideal $I$, is perfect, then we write
  $\omega_{R/I} = M^\vee$ and call it the {\bf canonical module} of
  $R/I$. If $R/I$ is perfect and $\omega_{R/I} \cong R/I$ (ignoring
  grading if it is present), then we say that $I$ is a {\bf Gorenstein
    ideal}. This is equivalent to the last term in the minimal free
  resolution of $R/I$ having rank 1.
\end{definition}

\begin{theorem}[Eagon--Northcott generic perfection] \label{thm:eagonnorthcott}
Let $R$ be a Noetherian ring and $M$ a perfect $R$-module of grade $g$, and let $\bF_\bullet$ be an $R$-linear free resolution of $M$ of length $g$.  Let $S$ be a Noetherian $R$-algebra. If $M \otimes_R S \ne 0$ and $\grade (M \otimes_R S) \ge g$, then $M \otimes_R S$ is perfect of grade $g$ and $\bF_\bullet \otimes_R S$ is an $S$-linear free resolution of $M \otimes_R S$. If $M \otimes_R S = 0$, then $\bF_\bullet \otimes_R S$ is exact.
\end{theorem}

See \cite[Theorem 3.5]{brunsvetter}. It is natural to ask what happens if the grade of $M \otimes_R S$ is some value less than $g$, and this can be answered by the Buchsbaum--Eisenbud acyclicity criterion \cite[Theorem 20.9]{eisenbud}.

\begin{remark} In particular, if $R$ is Cohen--Macaulay (for example,
  $R = K[x_1, \dots, x_n]$) then we can replace grade in the above
  theorem with codimension. It is often much easier to calculate
  codimension. We will use it as follows. We first construct graded
  minimal free resolutions of perfect modules $M$ over $A = K[x_1,
  \dots, x_n]$. Then we specialize the variables $x_i$ to elements of
  a Cohen--Macaulay $K$-algebra $S$ in such a way that the codimension
  of $M$ is preserved. Then the resulting specialized complex is still
  a resolution.
\end{remark}

\subsection{Examples.}

For this section, let $K$ be a commutative ring and let $E$ be a free
module of rank $N$. In the following examples, we will construct some
complexes that are functorial in $E$ and compatible with change of
rings. Two consequences of these properties is that the complex
carries an action of the general linear group $\GL(E)$ and that the
constructions make sense for vector bundles over an arbitrary scheme.

\begin{example} 
  Write $A = \Sym(E)$. We define a complex $\bF_\bullet$ by setting
  $\bF_i = \bigwedge^i E \otimes A(-i)$. The differential is defined
  by
  \begin{align*}
  \bigwedge^i E \otimes A(-i) &\to \bigwedge^{i-1} E \otimes A(-i+1)\\*
  e_1 \wedge \cdots \wedge e_i \otimes f &\mapsto \sum_{j=1}^i (-1)^j
  e_1 \wedge \cdots \widehat{e_j} \cdots \wedge e_i \otimes e_jf.
  \end{align*}
This is the {\bf Koszul complex}. It is a resolution of $K = A/\fm$, and is functorial with respect to $E$. See \cite[Chapter 17]{eisenbud} for basic properties.
\end{example}

\begin{example}[Buchsbaum--Eisenbud] \label{eg:buchsbaumeisenbud} We
  assume $N = 2n+1$ is odd. Set $A = \Sym(\bigwedge^2 E)$, which we
  can interpret as the coordinate ring of the space of all
  skew-symmetric matrices of size $2n+1$ with entries in $K$ if we fix
  a basis $e_1, \dots, e_{2n+1}$ of $E$. Let $\Phi$ be the generic
  skew-symmetric matrix of size $2n+1$ whose $(i,j)$ entry is $x_{ij}
  = e_i \wedge e_j \in A_1$. We construct a complex
  \[
  \bF_\bullet : 0 \to (\det E)^{\otimes 2} \otimes A(-2n-1) \to (\det
  E) \otimes E \otimes A(-n-1) \to \bigwedge^{2n} E \otimes A(-n) \to
  A.
  \]
  For $j =1, \dots, 2n+1$, let $e'_j = e_1 \wedge \cdots \hat{e_j}
  \cdots \wedge e_{2n+1}$. We also define $\Pf(\hat{\jmath})$ to be the
  Pfaffian of the submatrix of $\Phi$ obtained by deleting row and
  column $j$. Then we have
  \begin{align*}
    \bigwedge^{2n} E \otimes A(-n) &\xrightarrow{d_1} A\\*
    e'_j \otimes f &\mapsto \Pf(\hat{\jmath}) f,\\
    (\det E) \otimes E \otimes A(-n-1) &\xrightarrow{d_2}
    \bigwedge^{2n} E \otimes A(-n) \\*
    (e_1 \wedge \cdots \wedge e_{2n+1}) \otimes e_j \otimes f &\mapsto 
    \sum_{i=1}^{2n+1} (-1)^i e'_i \otimes x_{ij} f,\\
    (\det E)^{\otimes 2} \otimes A(-2n-1) &\xrightarrow{d_3} (\det E)
    \otimes E \otimes A(-n-1)\\*
    (e_1 \wedge \cdots \wedge e_{2n+1})^2 \otimes f &\mapsto (e_1
    \wedge \cdots \wedge e_{2n+1}) \otimes \sum_{j=1}^{2n+1} (-1)^j
    \Pf(\hat{\jmath}) e_j f.
  \end{align*}
  This is the {\bf Buchsbaum--Eisenbud complex}. It is a resolution of
  $A/I$ where $I$ is the ideal generated by the $2n \times 2n$
  Pfaffians of $\Phi$, and it is functorial with respect to
  $E$. Furthermore, $A/I$ is a free $K$-module, and $I$ is a
  Gorenstein ideal of codimension 3. We can identify $d_2$ with the
  map $\Phi$. Buchsbaum and Eisenbud showed that given a codimension 3
  Gorenstein ideal $I$, there is an $n$ such that its free resolution
  is a specialization of the above complex.  See
  \cite[Theorem 2.1]{buchsbaumeisenbud} for more details.
\end{example}

\begin{example}[J\'ozefiak--Pragacz] \label{eg:jozefiakpragacz} We
  assume $N = 2n$ is even. Again we set $A = \Sym(\bigwedge^2 E)$ and
  let $\Phi$ be the generic skew-symmetric matrix. We will give the
  resolution for the ideal generated by the Pfaffians of size
  $2n-2$. We just give the functorial terms in the complex when $K$
  contains the field of rational numbers (the definition of the Schur
  functors $\bS$ are given in \S\ref{section:schurfunctors}):
  \begin{align*}
    \bF_0 &= A \\
    \bF_1 &= \bigwedge^{2n-2} E \otimes A(-n+1) \\
    \bF_2 &= \bS_{2,1^{2n-2}} E \otimes A(-n) \\
    \bF_3 &= (\det E) \otimes S^2 E \otimes A(-n-1) \oplus (\det E)^2
    \otimes (S^2 E)^* \otimes A(-2n+1) \\
    \bF_4 &= (\det E) \otimes \bS_{2,1^{2n-2}} E \otimes A(-2n) \\
    \bF_5 &= (\det E)^2 \otimes \bigwedge^2 E \otimes A(-2n-1) \\
    \bF_6 &= (\det E)^3 \otimes A(-3n).
  \end{align*}
  When $K$ is an arbitrary commutative ring, the functors must be
  defined differently, but the ranks of the modules remain the
  same. We refer the reader to \cite{pragacz} for the
  details. 
\end{example}

\begin{example}[Goto--J\'ozefiak--Tachibana] \label{eg:GJT} We set $A
  = \Sym(S^2 E)$, which we can interpret as the coordinate ring of the
  space of symmetric matrices of size $N$. We let $\Phi$ be the
  generic symmetric matrix. We give the terms of the resolution of the
  ideal generated by the minors of size $N-1$:
  \begin{align*}
    \bF_0 &= A\\
    \bF_1 &= (\det E)^2 \otimes (S^2 E)^* \otimes A(-N+1)\\
    \bF_2 &= (\det E)^2 \otimes \ker(E \otimes E^* \xrightarrow{\rm eval} K) \otimes A(-N)\\
    \bF_3 &= (\det E)^2 \otimes \bigwedge^2 E \otimes A(-N-1).
  \end{align*}
See \cite[\S 3]{jozefiak} for details. Over a field of characteristic 0, the term $\det E \otimes \ker(E \otimes E^* \xrightarrow{\rm eval} K)$ can be replaced by the Schur functor $\bS_{2,1^{n-2}} E$ (see \S\ref{section:schurfunctors} for the definition).
\end{example}

\subsection{The geometric technique.} \label{sec:kempfcollapsing}

The material in this section is not logically necessary for the rest of the paper, but it is the main tool behind the work of Kra\'skiewicz--Weyman, so we include it for completeness. For a reference, see \cite[Chapter 5]{weyman}. Note that we have changed notation.

Let $K$ be a field, let $X$ be a projective $K$-variety, and let $V$
be a vector space. Suppose we are given a short exact sequence of
locally free sheaves
\[
0 \to \cS \to \ul{V} \to \cT \to 0.
\]
We let $p_1 \colon \ul{V} \to V$ and $p_2 \colon \ul{V} \to X$ be the
projection maps. Set $Y = p_1(\cS) \subset V$ and $A = \cO_V =
\Sym(V^*)$. Note that $Y$ is the affine cone over some projective
variety in $\bP(V)$. Also, let $\cE$ be any vector bundle on $X$.

\begin{theorem} There is a minimal $A$-linear complex $\bF_\bullet$
  whose terms are
  \[
  \bF_i = \bigoplus_{j \ge 0} \rH^j(X; \bigwedge^{i+j}(\cT^*) \otimes
  \cE) \otimes A(-i-j).
  \]
  Furthermore, $\rH_i(\bF_\bullet) = 0$ for $i>0$ and for $i \le 0$,
  we have
  \[
  \rH_i(\bF_\bullet) = \rR^{-i}(\cO_\cS \otimes_{\cO_{\ul{V}}}
  p_2^*\cE) = \rH^{-i}(X; \Sym(\cS^*) \otimes_{\cO_X} \cE).
  \]
\end{theorem}

In particular, if the higher direct images of $\cO_\cS \otimes p_2^*
\cE$ vanish, then $\bF_\bullet$ is a free resolution of the
pushforward. In the case that $\cE = \cO_X$, this pushforward is an
$A$-algebra. The vanishing of the higher direct images is an intrinsic
property of the variety $Y$. 

The idea behind this theorem is to start with an affine cone variety
$Y$ and to find $X$ and $\cS$ that fit into the above framework. Then
the theorem above gives a tool for calculating the minimal free
resolution of $Y$.

\begin{example}[Eagon--Northcott complex] Let $E$ and $F$ be vector
  spaces of dimensions $m$ and $n$ and assume that $m \ge n$. We set
  $V = \hom(E, F)$ and let $Y \subset V$ be the subvariety of linear
  maps of rank at most $n-1$.

  This fits into the previous setup by taking $X = \Gr(n-1,F) \cong
  \bP(F^*)$. This has a tautological exact sequence of vector bundles 
  \[
  0 \to \cR \to \ul{F} \to \cO(1) \to 0
  \]
  where $\cR = \{(x,W) \in F \times X \mid x \in W\}$. Then we can
  take $\cS = \cH om(\ul{E},\cR) = \ul{E}^* \otimes \cR$. We will see
  in \S\ref{section:bott} that the higher direct images of
  $\Sym(\cS^*)$ vanish and so $\bF_\bullet$ gives a minimal free
  resolution for $\cO_Y$.

  Since $\cT^* = \ul{E} \otimes \cO(-1)$, we have $\bigwedge^d \cT^* =
  \bigwedge^d \ul{E} \otimes \cO(-d)$. So we can calculate the terms
  of $\bF_\bullet$ explicitly. For $i>0$, we have
  \[
  \bF_i = \bigwedge^{n+i-1} E \otimes \det F^* \otimes (S^{i-1}F)^*
  \otimes A(-n-i+1).
  \]
  The differentials can be calculated by noting that they will
  preserve the natural $\GL(E) \times \GL(F)$ action on $V$ and $Y$
  and that equivariant maps of this form are unique up to a choice of
  scalar. In fact, this construction works with $K = \bZ$, in which
  case we get uniqueness of scalars up to a choice of sign. A
  multilinear generalization of this complex, constructed using
  similar ideas, can be found in \cite{beks}.
\end{example}

\section{Representation theory.} \label{section:reptheory}

\subsection{Schur functors.} \label{section:schurfunctors}

For the material in this section, see \cite[Chapter 2]{weyman}. What
we call $\bS_\lambda$ is denoted by $\bL_{\lambda'}$ there.

\begin{definition} A partition $\lambda$ is a decreasing sequence of
  positive integers $\lambda_1 \ge \lambda_2 \ge \cdots \ge
  \lambda_n$. We represent this as a Young diagram by drawing
  $\lambda_i$ boxes left-justified in the $i$th row, starting from top
  to bottom. The dual partition $\lambda'$ is obtained by letting
  $\lambda'_i$ be the number of boxes in the $i$th column of
  $\lambda$. Given a box $b = (i,j) \in \lambda$, its {\bf content} is
  $c(b) = j-i$ and its {\bf hook length} is $h(b) = \lambda_i - i +
  \lambda'_j - j + 1$. If we have a sequence $(i,i,\dots,i)$ repeated $j$ times, we abbreviate by the notation $(i^j)$.
\end{definition}

\begin{example} Let $\lambda = (4,3,1)$. Then $\lambda' =
  (3,2,2,1)$. The contents and hook lengths are given as follows:
  \[
  c: \tableau[scY]{0,1,2,3|-1,0,1|-2} \quad \quad h:
  \tableau[scY]{6,4,3,1|4,2,1|1} \qedhere
  \]
\end{example}

\begin{definition} Let $R$ be a commutative ring and $E$ a free
  $R$-module. Let $\lambda$ be a partition with $n$ parts and write $m
  = \lambda_1$. We use $S^n E$ to denote the $n$th symmetric power of
  $E$. The {\bf Schur functor} $\bS_\lambda(E)$ is the image of the
  map
  \begin{align*}
    \bigwedge^{\lambda'_1} E \otimes \cdots \otimes
    \bigwedge^{\lambda'_m} E \xrightarrow{\Delta} E^{\otimes
      \lambda'_1} \otimes \cdots \otimes E^{\otimes \lambda'_m} =
    E^{\otimes \lambda_1} \otimes \cdots \otimes E^{\otimes \lambda_n}
    \xrightarrow{\mu} S^{\lambda_1} E \otimes \cdots \otimes
    S^{\lambda_n} E,
  \end{align*}
  where the maps are defined as follows. First, $\Delta$ is the
  product of the comultiplication maps $\bigwedge^i E \to E^{\otimes
    i}$ given by $e_1 \wedge \cdots \wedge e_i \mapsto \sum_{w \in
    \Sigma_i} \sgn(w) e_{w(1)} \otimes \cdots \otimes e_{w(i)}$. The
  equals sign is interpreted as follows: pure tensors in $E^{\otimes
    \lambda'_1} \otimes \cdots \otimes E^{\otimes \lambda'_m}$ can be
  interpreted as filling the Young diagram of $\lambda$ with vectors
  along the columns, which can be thought of as pure tensors in
  $E^{\otimes \lambda_1} \otimes \cdots \otimes E^{\otimes \lambda_n}$
  by reading via rows. Finally, $\mu$ is the multiplication map
  $E^{\otimes i} \to S^i E$ given by $e_1 \otimes \cdots \otimes e_i
  \mapsto e_1 \cdots e_i$.

  In particular, note that $\bS_\lambda E = 0$ if the number of parts
  of $\lambda$ exceeds $\rank E$.
\end{definition}

\begin{example} Take $\lambda = (3,2)$. Then the map is given by
  \[
  \begin{array}{ll} (e_1 \wedge e_2) \otimes (e_3 \wedge e_4) \otimes
    e_5 & \mapsto {\tableau[scY]{e_1,e_3,e_5|e_2,e_4}} -
    {\tableau[scY]{e_2,e_3,e_5|e_1,e_4}} -
    {\tableau[scY]{e_1,e_4,e_5|e_2,e_3}} +
    {\tableau[scY]{e_2,e_4,e_5|e_1,e_3}}\\*
    & \mapsto (e_1e_3e_5 \otimes e_2e_4) - (e_2e_3e_5 \otimes e_1e_4)
    - (e_1e_4e_5 \otimes e_2e_3) + (e_2e_4e_5 \otimes e_1e_3).
  \end{array} \qedhere
  \]
\end{example}

\begin{theorem} The Schur functor $\bS_\lambda E$ is a free
  $R$-module. If $\rank E = n$, then
  \[
  \rank \bS_\lambda E = \prod_{b \in \lambda} \frac{n + c(b)}{h(b)}.
  \]
\end{theorem}

The construction of $\bS_\lambda E$ is functorial with respect to
$E$. This has two consequences: $\bS_\lambda E$ is naturally a
representation of $\GL(E)$, and we can also construct $\bS_\lambda
\cE$ when $\cE$ is a vector bundle.

If $\rank E = n$, then we have $\bS_{\lambda} E \otimes \det E = \bS_{(\lambda_1 + 1, \dots, \lambda_n + 1)} E$. Using this, it makes sense to define $\bS_\lambda E$ when $\lambda$ is any weakly decreasing sequence of integers. Furthermore, we have $\bS_\lambda(E^*) = \bS_{-\lambda_n, \dots, -\lambda_1} E$ \cite[Exercise 2.18]{weyman} and over a field of characteristic 0, the isomorphism $\bS_\lambda(E^*) = (\bS_\lambda E)^*$ \cite[Proposition 2.1.18]{weyman}.

\subsection{Descriptions of some homogeneous spaces.} \label{section:homogeneousspaces}

Let $E$ be a vector space of rank $N$. We let $\Fl(E)$ be the flag
variety of $E$. Its $K$-valued points are complete flags of subspaces
$E_\bullet : E_1 \subset E_2 \subset \cdots E_N = E$ such that $\rank
E_i = i$. The trivial bundle $\ul{E}$ contains a tautological flag of
subbundles $\cR_1 \subset \cR_2 \subset \cdots \subset \cR_N = \ul{E}$
where
\[
\cR_i = \{(x, E_\bullet) \in E \times \Fl(E) \mid x \in E_i\}.
\]
Given a subset $S \subset \{1, \dots, N-1\}$, we can also consider the
partial flag varieties $\Fl(S; E)$ whose $K$-valued points only
are partial flags whose ranks are the elements in $S$. Then $\ul{E}$ has a tautological partial flag of subbundles $\cR_i$ ($i \in S$).

Now assume that $E$ is equipped with a symplectic form $\omega$ and set $n = N/2$. We say that a subspace $U \subset E$ is isotropic if
$\omega(u,u') = 0$ for all $u,u' \in U$. Also, for any subspace $U$,
we set $U^\perp = \{x \in E \mid \omega(x,u) = 0 \text{ for all } u
\in U\}$. Then $\rank U + \rank U^\perp = \rank E$ and $U$ is
isotropic if and only if $U \subseteq U^\perp$.

We define the symplectic flag variety to be the subvariety
$\Fl_\omega(E)$ of the partial flag variety $\Fl(1,\dots,n; E)$
consisting of flags $E_1 \subset E_2 \subset \cdots \subset E_n$ such
that each $E_i$ is isotropic. We also let $\cR_1 \subset \cR_2 \subset
\cdots \subset \cR_n$ denote the restriction of the tautological
bundles to $\Fl_\omega(E)$. We can also define the subbundles
$\cR_i^\perp = \{(x, E_\bullet) \in E \times \Fl_\omega(E) \mid x \in
E_i^\perp\}$. Note that we have
\[
0 \subset \cR_1 \subset \cR_2 \subset \cdots \subset \cR_n =
\cR_n^\perp \subset \cR_{n-1}^\perp \subset \cdots \subset \cR_1^\perp
\subset \ul{E}.
\]
Given a subset $S \subset \{1, \dots, n\}$, we can also define partial
$\omega$-isotropic flag varieties $\Fl_\omega(S;E)$. When $S = \{i\}$
is a singleton, we also write $\Gr_\omega(i,E) = \Fl_\omega(\{i\}; E)$
and call it the $\omega$-isotropic Grassmannian. 

\subsection{Borel--Weil--Bott theorem.} \label{section:bott}

For the material in this section, see \cite[Chapter 4]{weyman}.

Let $S = \{i_1 < i_2 < \cdots < i_k\}$ be a subset of $\{1,\dots,N\}$
and consider the partial flag variety $\Fl(S; E)$. For each
$j=1,\dots,k+1$, let $\lambda^{(j)}$ be a weakly decreasing sequence
of integers of length $i_j - i_{j-1}$ (set $i_0 = 0$ and $i_{k+1} =
N$). Let $\lambda$ be the sequence obtained by concatenating
$\lambda^{(1)}, \lambda^{(2)}, \dots, \lambda^{(k+1)}$. Set
\[
\cR(\lambda) = \bigotimes_{j=1}^{k+1} \bS_{\lambda^{(j)}}((\cR_{i_j} /
\cR_{i_{j-1}})^*). 
\]

\begin{theorem}[Borel--Weil] If $\lambda$ is a weakly decreasing
  sequence of integers, then we have a $\GL(E)$-equivariant
  isomorphism
  \[
  \rH^0(\Fl(S;E); \cR(\lambda)) = \bS_\lambda(E^*),
  \]
  and all higher cohomology of $\cR(\lambda)$ vanishes.
\end{theorem}

Now we consider the analogue for the symplectic group. Let $\lambda$ be a partition with at most $n$ parts. We restrict the line bundle $\cR(\lambda)$ on the flag variety $\Fl(E)$ to the symplectic flag variety $\Fl_\omega(E)$ and denote its space of sections by
\begin{align} \label{eqn:symplecticschur}
\bS_{[\lambda]}(E) = \rH^0(\Fl_\omega(E); \cR(\lambda)).
\end{align}
By considering a relative situation, this definition makes sense for any vector bundle $E$ over a scheme $X$ which is equipped with a nondegenerate skew-symmetric $\cO_X$-bilinear form $\bigwedge^2 E^* \to \cO_X$. (More generally, the form could take values in a line bundle, but we will not use this generalization.) For all tautological subbundles $\cR$ on any symplectic flag variety, the restriction of the symplectic form to the quotient bundle $\cR^\perp / \cR$ is nondegenerate.

Now let $S = \{i_1 < \cdots < i_k\}$ be a subset of $\{1, \dots, n\}$
and let $\lambda^{(j)}$ be as before, except now we assume that
$\lambda^{(k+1)}$ only has nonnegative integers. We set
\[
\cR[\lambda] = \bS_{[\lambda^{(k+1)}]}(\cR_{i_k}^\perp / \cR_{i_k})
\otimes \bigotimes_{j=1}^k \bS_{\lambda^{(j)}}((\cR_{i_j} /
\cR_{i_{j-1}})^*).
\]

\begin{theorem}[Borel--Weil] 
  If $\lambda$ is a weakly decreasing sequence of nonnegative
  integers, then we have a $\Sp(E)$-equivariant isomorphism
  \[
  \rH^0(\Fl_\omega(S;E); \cR[\lambda]) = \bS_{[\lambda]} E
  \]
  and all higher cohomology of $\cR[\lambda]$ vanishes.
\end{theorem}

Now we discuss what happens when $\lambda$ is not a weakly decreasing
sequence of integers. For this, we now need to assume that the
characteristic of the ground field is 0. We first handle the partial
flag varieties.

Define the vector $\rho = (N-1, N-2, \dots, 1, 0)$. Given a
permutation $w \in \Sigma_N$, we define the dotted action of $w$ on a
sequence of integers $\alpha$ of length $N$ by
\[
w^\bullet(\alpha) = w(\alpha + \rho) - \rho.
\]
We define the length of a permutation $w$ to be $\ell(w) = \#\{(i, j)
\mid i < j,\ w(i) > w(j)\}$. Alternatively, let $s_i$ be the simple
transposition that swaps $i$ and $i+1$. Then $s_1, \dots, s_{N-1}$
generate $\Sigma_N$, and we could also define $\ell(w)$ to be the
minimal number $k$ so that we can write $w = s_{i_1} \cdots s_{i_k}$.

\begin{theorem}[Bott] \label{thm:glbott}
  Assume that the field $K$ has characteristic $0$ and let $\lambda$
  be as before. Then exactly one of the following two cases occurs:
  \begin{compactitem}
  \item There exists a non-identity $w \in \Sigma_N$ such that
    $w^\bullet(\lambda) = \lambda$. In this case, all cohomology of
    $\cR(\lambda)$ vanishes.
  \item There is a unique $w \in \Sigma_N$ such that
    $w^\bullet(\lambda) = \mu$ is a decreasing sequence. In this case,
    we have a $\GL(E)$-equivariant isomorphism
    \[
    \rH^{\ell(w)}(\Fl(S; E); \cR(\lambda)) = \bS_{\mu}(E^*)
    \]
    and all other cohomology of $\cR(\lambda)$ vanishes.
  \end{compactitem}
\end{theorem}

An example that uses the previous theorem is given in the proof of Theorem~\ref{thm:w39calc}.

Now we consider Bott's theorem for the symplectic flag varieties
$\Fl_\omega(S; E)$. Now we set $\rho = (N,
N-1, \dots, 2, 1)$. We replace the symmetric group $\Sigma_N$ with the
group of signed permutations $W = \Sigma_N \ltimes (\bZ/2)^N$, which
we think of as $N \times N$ signed permutation matrices.

We consider the generators $s_1, \dots, s_N$ for $W$. The meaning of
$s_1, \dots, s_{N-1}$ is the same as for the symmetric group
$\Sigma_N$, and $s_N$ is the diagonal matrix ${\rm diag}(1, \dots, 1,
-1)$. Now given $w \in W$, we define $\ell(w)$ to be the minimal
number $k$ so that we can write $w = s_{i_1} \cdots s_{i_k}$.  

Then the definition of the dotted action of $W$ remains the same, and
the analogue of Bott's theorem holds for the bundles $\cR[\lambda]$ on
$\Fl_\omega(S; E)$.

\begin{theorem}[Bott] 
  Assume that the field $K$ has characteristic $0$ and let $\lambda$
  be as before. Then exactly one of the following two cases occurs:
  \begin{compactitem}
  \item There exists a non-identity $w \in W$ such that
    $w^\bullet(\lambda) = \lambda$. In this case, all cohomology of
    $\cR[\lambda]$ vanishes.
  \item There is a unique $w \in \Sigma_N$ such that
    $w^\bullet(\lambda) = \mu$ is a weakly decreasing sequence of
    nonnegative integers. In this case, we have an
    $\Sp(E)$-equivariant isomorphism
    \[
    \rH^{\ell(w)}(\Fl_\omega(S; E); \cR[\lambda]) = \bS_{[\mu]}(E^*)
    \]
    and all other cohomology of $\cR[\lambda]$ vanishes.
  \end{compactitem}
\end{theorem}

The symplectic form gives $\Sp(E)$-equivariant isomorphisms $\bS_{[\mu]}(E^*) = \bS_{[\mu]} E$ for all $\mu$.

\subsection{Vinberg $\theta$-representations.} \label{section:vinberg}

For this section, we refer to \cite{kac1, kac, vinberg, vinberg2} for reference.

Let $X_n$ be a Dynkin diagram and let $\fg$ be the corresponding simple Lie algebra. Let us distinguish a node $x\in X_n$. Let $\alpha_k$ be a corresponding simple root in the root system $\Phi$ corresponding to $X_n$. The choice of $\alpha_k$ determines a $\bZ$-grading on $\Phi$ by letting the degree of a root $\beta$ be equal to the coefficient of $\alpha_k$ when we write $\beta$ as a linear combination of simple roots. On the level 
of Lie algebras this corresponds to a $\bZ$-grading
\[
{\fg}=\bigoplus_{i\in \bZ}\ {\fg}_i .
\]

We define the group $G_0 := (G ,G )\times {\bC}^*$ where $(G,G)$ is a 
connected semisimple group with the Dynkin diagram $X_n\setminus x$. 
A representation of type I is the representation of $G_0$ on ${\fg}_1$, and a representation of type II is what we get when we replace $X_n$ with an affine Dynkin diagram. This notation follows \cite[Proposition 3.1]{kac}.

\begin{remark}
In the case of a type II representation (i.e., when $X_n$ is an affine Dynkin diagram), each $\theta$-representation $(G_0,\fg_1)$ has a Chevalley isomorphism: there is a subspace $\fh \subset \fg_1$ and a complex reflection group $W \subset \GL(\fh)$ (defined as the normalizer of $\fh$ modulo the centralizer of $\fh$), which we call the {\bf graded Weyl group}, such that the restriction map is an isomorphism
\[
\Sym(\fg_1^*)^{(G,G)} \xrightarrow{\cong} \Sym(\fh^*)^W.
\]
In particular, $\Sym(\fg_1^*)^{(G,G)}$ is always a polynomial ring. The complex reflection groups were classified by Shephard--Todd in \cite{shephard} and we will refer to their numbering system when using names for these groups.
\end{remark}

In \cite{vinberg}, Vinberg gave a description of the $G_0$-orbits in the representations of type I in terms of conjugacy classes of nilpotent elements in $\fg$. Let $e\in {\fg}_1$
 be a nilpotent element in ${\fg}$. Consider the irreducible components 
of the intersection of the conjugacy class of $e$ in $\fg$
\[
C(e)\cap {\fg}_1 = C_1 (e)\cup\cdots\cup C_{n(e)}(e) .
\]
The sets $C_i (e)$ are clearly $G_0$-stable.  Vinberg's result 
shows that these are precisely the $G_0$-orbits in ${\fg}_1$.

\begin{theorem}[Vinberg] \label{thm:vinbergclassification}
The $G_0$-orbits of the action of $G_0$ on ${\fg}_1$ are the components $C_i (e)$, for all choices of the conjugacy classes $C(e)$ and all $i$, $1\le i\le n(e)$.
\end{theorem}

Theorem~\ref{thm:vinbergclassification} makes a connection between the orbits in ${\fg}_1$ and the nilpotent orbits in $\fg$.
The classification of nilpotent orbits in simple Lie algebras was obtained by Bala and Carter in the papers \cite{balacarter}. A good account of this theory is the book \cite{colmacg}. Here we recall that the nilpotent orbit of an element $e$ in a simple Lie algebra $\fg$ is characterized by the smallest Levi subalgebra $\mathfrak l$ containing $e$. One must be careful because sometimes $\mathfrak l$ is equal to $\fg$. If the element $e$ is a principal element in $\mathfrak l$, then this orbit is denoted by the Dynkin diagram of $\mathfrak l$ (but there might be different ways in which the root system $R({\mathfrak l })$ sits as a subroot system of $R({\fg})$). 

There are, however, the non-principal nilpotent orbits that are not contained in a smaller reductive Lie algebra $\mathfrak l$. These are called {\it the distinguished nilpotent orbits} and are described in \cite[\S\S 8.2--8.4]{colmacg}. They are characterized by their associated parabolic subgroups (as their Dynkin characteristics are even, \cite[\S 8]{colmacg}).
Let us remark that for Lie algebras of classical types, for type ${\rm A}_n$ the only distinguished nilpotent orbits are the principal ones, and for types ${\rm B}_n$, ${\rm C}_n$, ${\rm D}_n$ these are orbits corresponding to the partitions with different parts. For exceptional Lie algebras the distinguished orbits can be read off the tables in \cite[\S 8.4]{colmacg}.

Theorem~\ref{thm:vinbergclassification} is not easy to use because it is not very explicit. In the next section we describe a more precise method from another of Vinberg's papers \cite{vinberg2}.

\subsection{The Vinberg method for classifying orbits.}

In this section we describe the second paper of Vinberg \cite{vinberg2} in which he describes orbits of nilpotent elements in ${\fg}_1$.
Similar to the Bala--Carter classification, the nilpotent elements in ${\fg}_1$ are described by means of some graded subalgebras of $\fg$. We need some preliminary notions.

All Lie algebras $\fg$ we will consider will be Lie algebras of some algebraic group $G$.

Let $(X_n ,\alpha_k )$ be one of the representations from the previous section. As before, it defines the grading
\[
{\fg}=\bigoplus_{i\in{\bZ}} {\fg}_i
\]
where ${\fg}_i$ is the span of the roots which, written as a combination of simple roots, have $\alpha_k$ with coefficient $i$. The component ${\fg}_0$ contains a Cartan subalgebra. $G_0$ denotes the connected closed subgroup of $G$ corresponding to ${\fg}_0$.

In the sequel, $Z(x)$ denotes the centralizer of an element $x\in G$ and $Z_0(x) = Z(x)\cap G_0$, and $\fz$, ${\mathfrak z}_0$ denote the corresponding Lie algebras. Similarly, $N(x)$ denotes the normalizer of an element $x\in G$ and $N_0(x) = N(x)\cap G_0$.

We let $R({\fg})$ denote the set of roots of $\fg$, and $\Pi({\fg})$ denotes a set of simple roots.

\begin{definition} A graded Lie subalgebra $\fs \subset \fg$ is {\bf regular} if it is normalized by a maximal torus in ${\fg}_0$. A reductive graded Lie algebra $\fs \subset \fg$ is {\bf complete} if it is not a proper graded Lie subalgebra of any regular reductive $\bZ$-graded Lie algebra of the same rank.
\end{definition}

\begin{definition} A $\bZ$-graded Lie algebra $\fg$ is {\bf locally flat} if any of the following equivalent conditions is satisfied, for $e$ a point in general position in ${\fg}_1$:
\begin{compactenum}[(1)]
\item the subgroup $Z_0(e)$ is finite,
\item $\fz_0(e)=0$,
\item $\dim {\fg}_0 = \dim {\fg}_1$. \qedhere
\end{compactenum}
\end{definition}

Fix a nonzero nilpotent element $e\in {\fg}_a$ and choose some maximal torus $H$ in $N_0(e)$. Its Lie algebra $\mathfrak h$ is {\bf the accompanying torus} of the element $e$. We denote by $\phi$ the character of the torus $H$ defined by the condition
\[
[u,e]=\phi(u)e
\]
for $u\in{\mathfrak h}$. Consider the graded Lie subalgebra ${\fg}({\mathfrak h}, \phi )$ of $\fg$ defined by
\[
{\fg}({\fh}, \phi ) =\bigoplus_{i\in\bZ} {\fg}({\mathfrak h}, \phi )_i
\]
where
\[
{\fg}({\mathfrak h}, \phi )_i=\lbrace x\in {\fg}_{ia} \mid [u,x]=i\phi(u)x \text{ for all } u\in H \rbrace .
\]

\begin{definition} The {\bf support} $\mathfrak s$ of the nilpotent element $e\in {\fg}_a$ is the commutator subalgebra of ${\fg}({\mathfrak h}, \phi )$ considered as a $\bZ$-graded Lie algebra.
\end{definition}

Clearly $e\in {\mathfrak s}_1$. We are ready to state the main theorem of \cite {vinberg2}. 

\begin{theorem}[Vinberg] \label{thm:2vinberg} The supports of nilpotent elements of the space ${\fg}_i$ are exactly the complete regular locally flat semisimple $\bZ$-graded subalgebras of the algebra $\fg$.
The nilpotent element $e$ can be recovered from the support subalgebra $\mathfrak s$ as the generic element in ${\mathfrak s}_1$.
\end{theorem}

It follows from the theorem that the nilpotent element $e$ is defined uniquely (up to conjugation by an element of $G_0$) by its support. This means it is enough to classify the regular semisimple $\bZ$-graded subalgebras $\mathfrak s$ of $\fg$.
 
Let us choose a maximal torus $\ft \subset {\fg}_0$. The $\bZ$-graded subalgebra $\mathfrak s$ is {\bf standard} if it is normalized by $\mathfrak t$, i.e., if for all $i\in\bZ$ we have
\[
[{\mathfrak t}, {\mathfrak s}_i]\subset {\mathfrak s}_i .
\]

Vinberg also proves that every $\bZ$-graded subalgebra $\mathfrak s$ is conjugated to a standard subalgebra by an element of $G_0$. Moreover, he shows that if two standard $\bZ$-graded subalgebras are conjugated by an element of $G_0$, then they are conjugated by an element of $N_0 ({\mathfrak t})$. This gives combinatorial method for classifying regular semisimple $\bZ$-graded subalgebras of $\fg$.

Let $\fs$ be a standard semisimple $\bZ$-graded subalgebra of $\fg$. The grading on the subalgebra $\mathfrak s$ defines a degree map $\deg \colon R({\mathfrak s)}\rightarrow\bZ$.
For a standard $\bZ$-graded subalgebra $\mathfrak s$ we also get the map 
\[
f \colon R({\mathfrak s)}\rightarrow R({\fg}).
\]

The map $f$ has to be {\it additive}, i.e., it satisfies
\begin{align*}
f(\alpha +\beta ) &= f(\alpha )+f(\beta ) \quad \forall\alpha ,\beta\in R({\mathfrak s}),\\*
f(-\alpha  ) &= -f(\alpha ) \quad \forall\alpha \in R({\mathfrak s}).
\end{align*}
Moreover we have

\begin{proposition} \label{condmapf} The map $f$ satisfies the following properties:
\begin{compactenum}[\rm (a)]
\item
\[
{{(f(\alpha ),f(\beta ))}\over{(f(\alpha ),f(\alpha ))}}= {{(\alpha ,\beta )}\over{(\alpha ,\alpha )}}\quad \forall\alpha, \beta\in R({\mathfrak s}).
\]
\item $f(\alpha )-f(\beta )\notin R({\fg})\quad \forall\alpha ,\beta\in \Pi({\mathfrak s})$,
\item $\deg f(\alpha )=\deg \alpha,\quad \forall\alpha\in \Pi({\mathfrak s})$.
\end{compactenum}
Conversely, every map satisfying these conditions defines a standard regular $\bZ$-graded subalgebra ${\mathfrak s}$ of ${\fg}$.
\end{proposition}

\begin{remark} The subalgebra $\mathfrak s$ corresponding to the map $f$ is complete if and only if there exists an element $w$ in the Weyl group $W$  of $\fg$
such that $wf(\Pi({\mathfrak s}))\subset\Pi({\fg})$ (see \cite[p.25]{vinberg2}).
\end{remark}

Theorem \ref{thm:2vinberg} means that in order to classify the nilpotent elements $e\in {\fg}_1$ we need to classify the possible maps $f$
corresponding to its support, i.e., the corresponding complete regular $\bZ$-graded subalgebra $\mathfrak s$.
Since we are interested in the nilpotent elements $e\in {\fg}_1$, we need to classify the maps $f$ for which $\deg( f(\alpha ))\in\lbrace 0,1\rbrace$ for every $\alpha\in \Pi({\mathfrak s})$.

\subsection{Example: $\bigwedge^3 \bC^7$.}

Let $X_n={\rm E}_7$, and $\alpha_k=\alpha_2$ in Bourbaki numbering.

The graded Lie algebra corresponding to the resulting grading is 
\[
{\fg}(E_7)= {\fg}_{-2}\oplus {\fg}_{-1}\oplus {\fg}_0\oplus {\fg}_1\oplus {\fg}_2
\]
with $G_0= \GL_7(\bC)$, ${\fg}_0= \mathfrak{gl}_7(\bC)$, 
${\fg}_1=\bigwedge^3\bC^7$, ${\fg}_2=\bigwedge^6\bC^7$.

We choose a basis $\lbrace e_1,\ldots, e_7\rbrace$ of $\bC^7$.
The weight vectors in $\fg_1=\bigwedge^3\bC^7$ are the vectors $e_i\wedge e_j\wedge e_k$ of weight 
$\eps_i +\eps_j+\eps_k$ for $1\le i<j<k\le 7$. 

There is a natural bijection $\Sigma$ of these weight vectors with the positive roots with label $1$ on the node $\alpha_2$. In order to describe it
we identify the subroot system of ${\fg}_0$ with the root system of type ${\rm A}_6$ with the simple roots $\beta_1 ,\ldots ,\beta_6$ (corresponding to 
permutations $(1,2), (2,3),\ldots ,(6,7)$ respectively) by identifying $\beta_1$ with $\alpha_1$ and $\beta_s$ with $\alpha_{s+1}$ for $2\le s\le 6$.
We denote this identification map by $\Lambda$. We set $\Sigma(e_5\wedge e_6\wedge e_7)=\alpha_2$ and extend it to a unique bijection satisfying $\Sigma(e_i\wedge e_j\wedge e_k+\beta_s)=\Sigma(e_i\wedge e_j\wedge e_k)+\Lambda(\beta_s)$ for every simple root $\beta_s$, for  $1\le s\le 6$. 

The invariant scalar product on ${\mathfrak h}$ restricted to the roots from ${\fg}_1$ and transferred by $\Sigma$ is
\[
(e_i\wedge e_j\wedge e_k, e_p\wedge e_q\wedge e_r)= \# (\lbrace i,j,k\rbrace\cap\lbrace p,q,r\rbrace)-1.
\]

This can be checked by observing that both the invariant scalar product and the one defined by the formula above are
invariant with respect to the Weyl group of ${\fg}_0$, i.e., the permutation group $S_7$, and checking the equality of one scalar product directly.

Next using the table from \cite[p.248]{spaltenstein} we see that the general element of $\bigwedge^3 \bC^7$ has the Bala--Carter label ${\rm A}_2+3{\rm A}_1$. Notice that the labeling of this orbit given on \cite[p.130]{colmacg} has weighted Dynkin diagram with labels $0$ on all nodes except for $\alpha_k$ and label $2$ on the node $\alpha_k$ (this also happens in the other cases except the case $({\rm E}_6,\alpha_3)$ where the weighted Dynkin diagram of the nilpotent orbit intersecting ${\fg}_1$ in the open orbit has labelings $1$ at $\alpha_3$ and $\alpha_5$ and zero elsewhere). The support subalgebras of the smaller orbits in ${\fg}_1$ have Vinberg labels exactly matching the Bala--Carter labels that can be read off the Spaltenstein tables. In our analysis we checked also that the other types of support algebras given in \cite{vinberg2} do not exist.

The combinatorial analysis is not difficult. We just give a few examples. Let us identify the support subalgebra of type ${\rm A}_2+3{\rm A}_1$.
According to Proposition~\ref{condmapf} we need to find all choices (up to the action of $S_7$) of five weight vectors in $\bigwedge^3\bC^7$ whose pairwise  scalar products match those of the five simple roots of the root system ${\rm A}_2+3{\rm A}_1$. It is clear that for simple roots of the system ${\rm A}_2$ we can choose $e_1\wedge e_2\wedge e_3$ and $e_4\wedge e_5\wedge e_6$. Then we need a choice of three subsets $[i,j,k]$  which have one element intersections
with both $[1,2,3], [4,5,6]$ and with each other. Up to the permutation group $S_7$ it is clear that there is only one choice $[1,4,7], [2,5,7], [3,6,7]$.
One can check that the resulting orbit is open by calculating its dimension. This is an easy linear algebra exercise since the tangent space is obtained by hitting this element with all vectors in $\fg\fl_7(\bC)$.

Similarly, looking at the subsystems of type $3{\rm A}_1$ we need choices of three subsets $[i,j,k]$ such that every pair of subsets intersects in one element. This element can be common to all three subsets or not, which gives two possibilities $\{[1,2,3], [1,4,5], [1,6,7]\}$ and $\{[1,2,3], [1,4,5], [2,4,6]\}$ up to the $S_7$-action. 

Finally, looking at the possibilities for $4{\rm A}_1$ we notice that the triple $[1,2,3], [1,4,5], [1,6,7]$ can be only complemented in eight ways equivalent to $[2,4,6]$. The triple $[1,2,3], [1,4,5], [2,4,6]$ can be complemented in three ways equivalent to $[1,6,7]$. But these two choices of roots of $4{\rm A}_1$ are the same. We conclude that there is only one orbit of type $4{\rm A}_1$ in $\bigwedge^3 \bC^7$.

A subtle point is checking that the obtained support algebra is complete. In the case of $\bigwedge^3\bC^7$ it also follows from the fact that all the representatives give different orbits. We omit this.

The result of the analysis is presented in the following table (writing $[i,j,k]$ for $e_i\wedge e_j\wedge e_k$).

\[
\begin{array}{cccc} \text{Label} & \fs & \dim & \text{Representative}\\ 
0&0&0&0\\
1&{\rm A}_1&13&[123]\\
2&2{\rm A}_1&20&[123]+[145]\\
3&3{\rm A}_1&21&[123]+[145]+[167]\\
4&3{\rm A}_1&25&[123]+[145]+[246]\\
5&{\rm A}_2&26&[123]+[456]\\
6&4{\rm A}_1&28&[123]+[145]+[167]+[357]\\
7&{\rm A}_2+{\rm A}_1&31&[123]+[456]+[147]\\
8&{\rm A}_2+2{\rm A}_1&34&[123]+[456]+[147]+[257] \\
9&{\rm A}_2+3{\rm A}_1&35&[123]+[456]+[147]+[257]+[367]  
\end{array}
\]
The containment diagram is
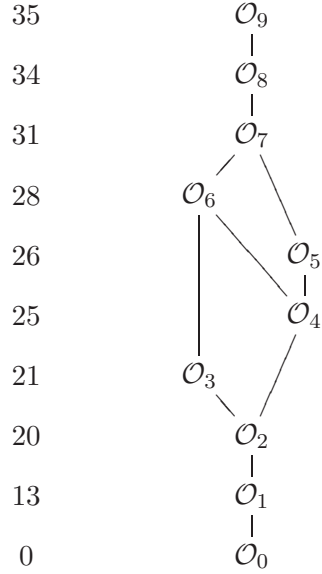
\[
\xy
(15,0)*+{{\cO}_{0}}="o0";%
(15,8)*+{{\cO}_{1}}="o1";%
(15,16)*+{{\cO}_2}="o2";%
(8,24)*+{{\cO}_3}="o3";%
(22,32)*+{{\cO}_4}="o4";%
(22,40)*+{{\cO}_5}="o5";%
(8,48)*+{{\cO}_6}="o6";%
(15,56)*+{{\cO}_7}="o7";%
(15,64)*+{{\cO}_8}="o8";%
(15,72)*+{{\cO}_9}="o9";%
(-15,0)*{0};%
(-15,8)*{13};%
(-15,16)*{20};%
(-15,24)*{21};%
(-15,32)*{25};%
(-15,40)*{26};%
(-15,48)*{28};%
(-15,56)*{31};%
(-15,64)*{34};%
(-15,72)*{35};%
{\ar@{-} "o0"; "o1"};%
{\ar@{-} "o1"; "o2"};%
{\ar@{-} "o3"; "o2"};%
{\ar@{-} "o4"; "o2"};%
{\ar@{-} "o4"; "o5"};%
{\ar@{-} "o3"; "o6"};%
{\ar@{-} "o4"; "o6"};%
{\ar@{-} "o5"; "o7"};%
{\ar@{-} "o6"; "o7"};%
{\ar@{-} "o7"; "o8"};%
{\ar@{-} "o8"; "o9"};%
\endxy
\]
and it happens to be the same as the corresponding containment of the nilpotent orbits in ${\fg}_1$.
 
We can interpret the orbits in terms of skew-symmetric tensors as follows.
Let $A=\Sym(\bigwedge^3{\bC^7}^*)$ be the coordinate ring of our representation.
The orbit $\cO_1$ is the orbit of the highest weight vector, and its closure is the set of decomposable tensors, i.e., the tensors $t=v_1\wedge v_2\wedge v_3$ ($v_1, v_2, v_3\in\bC^7$).
There is an $\SL_7(\bC)$-invariant $\Delta$ of degree $7$, the hyperdiscriminant. This is the projective dual variety of the decomposable tensors in ${\bC^7}^*$. In fact, we can canonically identify the orbits in $\bC^7$ with those in its dual, so it makes sense to consider the projective dual $\ol{\cO}^\vee$ of an orbit closure $\ol{\cO}$ as a subset of $\bC^7$ itself. By the {\bf rank} of a tensor $t\in\bigwedge^3{\bC}^7$ we mean the subspace rank, i.e., the minimal number $s$ such that there
exists a subpace $V$ of dimension $s$ in ${\bC}^7$ such that $t\in\bigwedge^3 V\subset\bigwedge^3{\bC}^7$.
An orbit is {\bf degenerate} if it consists of tensors of rank $\le 6$. Such orbits correspond to $\GL_6(\bC)$-orbits in $\bigwedge^3{\bC}^6$.
By a $1$-decomposable tensor we mean a tensor $t=v\wedge q$ with $v\in{\bC}^7$, $q\in\bigwedge^2{\bC}^7$.
Finally in the projective picture we interpret orbit closures as secant and tangential varieties of the orbit $\cO_1$. We denote 
by $\sigma_k({\cO}_1)$ the $k$th secant variety of $\cO_1$ and by $\tau({\cO}_1)$ the tangential variety of $\cO_1$.

With this terminology, we have the following table describing the orbits.
  
~

\begin{tabular}{lll}\text{Number} & \text{Projective picture} & \text{Tensor picture}\\ 
0&0&0\\
1& cone($\Gr(3,7)$)& decomposable tensors\\
2&&tensors of rank $\le 5$\\
3& $\sigma_2(\ol{\cO}_1)^\vee$ &1-decomposable\ tensors\\
4&$\tau(\overline\cO_1)$&degenerate and zero  hyperdiscriminant for $\bigwedge^3 \bC^6$\\
5&$\sigma_2 (\overline\cO_1)$&tensors\ of\ rank $\le 6$\\
6&$J(\overline\cO_1 ,\tau(\overline\cO_1))$&polarizations of hyperdiscriminant for $\bigwedge^3 \bC^6$\\
7&$\sigma_3 (\overline\cO_1)$& singular locus of $\ol{\cO}_1$ \\
8& $\ol{\cO}_1^\vee$ &hyperdiscriminant is zero \\
9&&generic  
\end{tabular}

\begin{remark}
The geometric description of these orbits is also considered in \cite[\S 5]{holweck} (he also considers $\bigwedge^3 \bC^6$ and $\bigwedge^3 \bC^8$). The classification of orbits of $\bigwedge^3 K^7$ is studied for many kinds of fields $K$ (including algebraically closed fields of positive characteristic and finite fields) in \cite{ch}. In particular, the classification of orbits is independent of characteristic if the field is algebraically closed.
\end{remark}

We will describe in detail the non-degenerate orbit closures in $\bigwedge^3\bC^7$.
These are the orbits $\overline{\cO_9}$, $\overline{\cO_8}$. $\overline{\cO_7}$, $\overline{\cO_6}$ and $\overline{\cO_3}$.
The first of these is generic so there is not much to say. We also describe the generic degenerate orbit of tensors of rank $\le 6$.

We use the usual notation. Let $A=\Sym(\bigwedge^3{\bC^7}^*)$ and $\lambda$ is notation for $\bS_\lambda({\bC^7}^*)$. Also, let $x_1, \dots, x_7$ be a basis of ${\bC^7}^*$ dual to the basis $e_1, \dots, e_7$ of $\bC^7$. We will also describe vector bundle desingularizations for these orbit closures. The bundles $\eta$ and $\xi$ correspond to $\cS^*$ and $\cT^*$ in the notation of \S\ref{sec:kempfcollapsing}.

\begin{itemize}[$\bullet$]

\item  The hyperdiscriminant orbit $\cO_{8}$. 

This is the hypersurface given by the tensors with vanishing hyperdiscriminant $\Delta$. The orbit closure ${\overline\cO}_8$ is characterized 
(set-theoretically) by the condition that the determinant of the multiplication map
\[
\bigwedge^5{\bC^7}^*\otimes A(-1)\rightarrow\bigwedge^2{\bC^7}^*\otimes A
\]
given by multiplication is zero. The determinant of this matrix is equal to $\Delta^3$.

\item  The codimension $4$ orbit $\cO_{7}$.

This orbit closure is the singular locus of the hyperdiscriminant orbit $\overline{\cO_8}$.

We can find a desingularization by a vector bundle over $G/P = \Fl(2,6;\bC^7)$. The bundle $\xi \subset \bigwedge^3 \ul{\bC^7}^*$ is induced from the largest $P$-submodule of $\bigwedge^3 {\bC^7}^*$ which does not contain $x_1 \wedge x_2 \wedge x_7$ and $x_2 \wedge x_5 \wedge x_7$. The bundle $\eta$ has rank $17$, so the dimension of the desingularization is $17+14=31$ as needed. One gets a very nice complex describing the resolution of $\bC[\ol{\cO}_7]$. The terms of the complex $\bF(7)_\bullet$ are as follows
\[
0\rightarrow (6,5^6)\rightarrow (5^2,4^5)\rightarrow (4,3^5,2)\rightarrow(3^4,2^3)\rightarrow (0^7).
\]
The orbit closure is normal and has rational singularities and the complex $\bF(7)_\bullet$ is pure.

\item The codimension $7$ orbit $\cO_{6}$. 

We can find a desingularization by a vector bundle over $G/P = \Fl(1,4;\bC^7)$. The bundle $\xi \subset \bigwedge^3\ul{\bC^7}^*$ is induced from the $P$-submodule of $\bigwedge^3 {\bC^7}^*$ which does not contain $x_1 \wedge x_4 \wedge x_7$ and $x_2 \wedge x_3 \wedge x_4$. The bundle $\eta$ has rank $13$, so the dimension of the desingularization is
$13+15=28$ as needed. The terms in the resulting complex $\bF(6)_\bullet$ are
\begin{align*}
0\rightarrow (7^6,6)\rightarrow (7,6^5,5)\rightarrow (6^2,5^4,4)
\rightarrow (5^3,4^3,3)\rightarrow (4^4,3^2,2)
\rightarrow (3^5,2,1)\rightarrow (2^6,0)\rightarrow (0^7).
\end{align*}
The orbit closure is normal, with rational singularities and the complex $\bF(6)_\bullet$ is pure.

\item The generic degenerate orbit closure $\overline{\cO}_5$ of tensors of rank $\le 6$ (codimension $9$).

This orbit closure has a desingularization $Z(5)$ that lives on the Grassmannian $\Gr(1,{\bC^7}^*)$. Denoting the tautological bundles 
 $\cR$, $\cQ$ ($\rank {\cR}=1$, $\rank \cQ=6$), we have
$\xi ={\cR}\otimes\bigwedge^2 {\cQ}$. It is normal and has rational singularities. Calculating the resolution is straightforward, as $\xi$ is irreducible.

\item The orbit $\cO_{3}$ of $1$-decomposable tensors (codimension $14$). 

This orbit closure is the set of tensors $t\in\bigwedge^3\bC^7$ that can be expressed as $t=\ell\wedge\overline t$ where $\ell \in \bC^7$, ${\overline t}\in\bigwedge^2 \bC^7$.
The desingularization $Z(3)$ lives on the Grassmannian $\Gr(6,{\bC^7}^*)$. Denoting the tautological bundles as $\cR$, $\cQ$ ($\rank {\cR}=6$, $\rank \cQ=1$), we have $\xi =\bigwedge^3{\cR}$. It is normal and has rational singularities.
Calculating the resolution is straightforward, as $\xi$ is irreducible. The defining ideal is generated by the representation $(2^3,1^3,0)$ in degree $3$.
\end{itemize}

\begin{remark} \label{rmk:kummerflag}
Since we will need it later on, we ask, for a fixed $v \in \bigwedge^3 \bC^7$, how many lines $\ell \subset \bC^7$ there are such that the image of $v$ in $\bigwedge^3 (\bC^7/\ell)$ is a pure tensor. This only depends on the orbit, so we can study specific representatives.

No such line exists for vectors in the generic orbit or codimension 1 orbit. A representative for the codimension 4 orbit is
\[
  e_1 \wedge e_2 \wedge e_3 + e_4 \wedge e_5 \wedge e_6 + e_1 \wedge
  e_4 \wedge e_7.
  \]
This is a pure tensor in $\bigwedge^3 (\bC^7/\ell)$ exactly for $\ell = \langle e_1 \rangle$ and $\ell = \langle e_4 \rangle$. A representative for the codimension 7 orbit is
  \[
  e_1 \wedge e_2 \wedge e_3 + e_1 \wedge e_4 \wedge e_5 + e_1 \wedge
  e_6 \wedge e_7 + e_3 \wedge e_5 \wedge e_7.
  \]
This is a pure tensor in $\bigwedge^3 (\bC^7/\ell)$ exactly for $\ell = \langle e_1 \rangle$.
\end{remark}

\section{Some geometry.} \label{section:curves}

\subsection{Abelian varieties.}

The following result is most likely well-known, but we could not find it in the literature. The main points of the proof were communicated to us by Damiano Testa.

\begin{theorem} \label{thm:cohomologyabelian} Let $X$ be a
  $g$-dimensional geometrically connected projective nonsingular
  variety over a field $K$ of characteristic $0$. If $\omega_X \cong
  \cO_X$ and $\dim \rH^1(X; \cO_X) = g$, then $X$ is a torsor over an
  Abelian variety (namely, its Albanese variety).
\end{theorem}

There are two reasons that we emphasize the fact that $X$ is only a torsor: in our applications of this theorem in the later sections, there may be no natural choice of base point (which is relevant when working over non-algebraically closed fields), and in later work we will be interested in working over families (in which case the existence of a section may be more subtle).

\begin{proof}
  First, extend scalars to the algebraic closure $\ol{K}$ of $K$. By
  \cite[Corollary 2]{kawamata}, $X_{\ol{K}}$ is birationally
  equivalent to its Albanese variety $X_{0,\ol{K}}$, let $f \colon
  X_{\ol{K}} \to X_{0,\ol{K}}$ be a birational morphism. In fact, $f$
  can be (uniquely) extended to a morphism on all of $X_{\ol{K}}$
  \cite[Theorem 4.9.4]{birkenhake}. One has an induced map $\rd f$ on
  cotangent bundles. The determinant of $\rd f$ is a map between the
  canonical bundles, which are trivial by assumption. The
  endomorphisms of the trivial bundle must be scalars since we assumed
  that $X_{\ol{K}}$ is projective, so $\det \rd f$ is a scalar. This
  scalar is nonzero since $f$ is generically an \'etale morphism. So
  in fact $f$ is \'etale (and hence open). Also $f$ is proper (and
  hence closed), so $f$ is an \'etale covering, which implies that
  $X_{\ol{K}}$ is an Abelian variety. Furthermore, $f$ must be an
  isomorphism since it is birational.

  Hence, choosing any point $P \in X(\ol{K})$, we have a
  $\ol{K}$-isomorphism $X_{0,\ol{K}} \to X_{\ol{K}}$ via $x \mapsto
  P+x$. This map descends to a $K$-rational map $Y \to X$ where $Y$ is
  a $K$-rational $X_0$-torsor given by the cocycle $\gamma \mapsto
  {}^\gamma P - P$, which gives the claim.
\end{proof}

\begin{remark} If we drop the assumption that $K$ be of characteristic 0, then this theorem already fails for $g=2$. In particular, it is valid if the characteristic is different from 2 or 3, but in these small characteristics, there are new exotic examples, known as quasi-hyperelliptic surfaces which come from the Bombieri--Mumford classification of surfaces (the quasi-hyperelliptic surfaces have the property that their Picard varieties are non-reduced), see \cite[p.25, Table]{bomb}.
\end{remark}

Given a smooth curve $C$ of genus $g$, we let $\Jac(C)$ denote its Jacobian, which is an Abelian variety of dimension $g$ (see \cite[Chapter 11]{birkenhake} for an analytic construction of $\Jac(C)$).

\subsection{Moduli space of vector bundles.} \label{sec:vectorbundles}

The material in this section is provided for convenience and informative purposes, since later we will see some examples of the moduli spaces discussed in this section (see Example~\ref{eg:ramanan}, Remark~\ref{rmk:coblequartic}, and \S\ref{sec:spin16}). We refer the reader to \cite{beauville:vb} for a more in-depth survey of the following.

Let $C$ be a smooth curve of genus $g \ge 2$. There is a moduli space
$SU_C(n,d)$ which parametrizes rank $n$ semistable vector bundles of
degree $d$ on $C$. It has dimension $n^2(g-1) + 1$. Let $SU_C(n,L)$
denote the moduli space of rank $n$ semistable vector bundles on $C$
with determinant equal to $L$. We write $SU_C(n)$ for
$SU_C(n,\cO_C)$. Then $SU_C(n)$ has dimension $(g-1)(n^2-1)$ and is
Gorenstein and has rational singularities. The Picard group of
$SU_C(n)$ is infinite cyclic. Let $\cL$ be its ample generator, which
we call the theta divisor. The canonical bundle is $\cL^{-2n}$.

Now focus on $n=2$. If $g=2$, then $SU_C(2) \cong \bP^3$, and for
$g>2$, the singular locus of $SU_C(2)$ is bundles of the form $L
\oplus L^{-1}$ and is naturally identified with the Kummer quotient of
the Jacobian of degree $g-1$ line bundles on $C$. Also,
$\rh^0(SU_C(n); \cL) = 2^g$ and the map given by $\cL$ is an embedding
if $C$ is not hyperelliptic. Furthermore, the restriction of $\cL$ is
a $(2, \dots, 2)$-polarization. See \cite{brivio} and \cite{geemen}
for more details.

Finally, we state the Verlinde formula, which gives the dimension of
the space of sections of powers of $\cL$. For $n=2$, we have 
\[
\rh^0(SU_C(2); \cL^k) = \left( \frac{k+2}{2} \right)^{g-1}
\sum_{j=1}^{k+1} (\sin\frac{\pi j}{k+2})^{-2g+2},
\]
see \cite[\S 5]{beauville:verlinde} for the case of general $n$.

\subsection{Degeneracy loci.}

The main idea of the paper is in the following construction.

\begin{construction} \label{construct:main}
Start with a Vinberg representation $(G,U)$ of affine type (see \S\ref{section:vinberg}). Choose a parabolic subgroup $P$ of $G$. We can realize $U$ as the sections of a homogeneous bundle $\cU$ over the homogeneous space $G/P$ using \S\ref{section:bott}.  In each case that we consider, the fibers of $\cU$ can naturally be interpreted as another Vinberg representation $(G',U')$ of finite type (more specifically, $G'$ will be the Levi subgroup of the stabilizer of the point where the fiber lives). We use information about the orbit closures in $(G',U')$ and patch them together to get subvarieties $\cY$ of the total space of $\cU$.

Given a section $v \in \rH^0(X; \cU)$, we consider the subvarieties $v(G/P) \cap \cY$, and in particular, when the grade of the ideal sheaf does not change. In all of the orbit closures in $U'$ of relevance, we give free resolutions for their coordinate rings and some related modules. Then this gives a locally free resolution of $v(X) \cap \cY \subset v(G/P) \cong G/P$ via Theorem~\ref{thm:eagonnorthcott}, and this will allow us to read off properties of this variety. In particular, we can try to use this resolution to calculate the canonical sheaf of $v(G/P) \cap \cY$ and the cohomology of its structure sheaf. 
\end{construction}

In all cases, we will find a choice of $P$ so that one of the degeneracy loci (or a variety closely related to it) is a torsor over an Abelian variety. We will also explore what happens when we vary the choice of $P$. In some cases, we are able to establish a direct link between these other degeneracy loci and the torsor via some classical geometric constructions (such as projective duality).

\begin{lemma} Let $Y^\circ$ be a $G'$-orbit in $U'$ and let $\cY^\circ$ be the union of these $G'$-orbits over all fibers. There is a nonempty open subset $U_Y^{\rm gen} \subset U$ such that either $\codim(v(G/P) \cap \cY^\circ, v(G/P)) = \codim(\cY^\circ, \cU)$ for all $v \in U_Y^{\rm gen}$, or $v(G/P) \cap \cY^\circ = \emptyset$ for all $v \in U_Y^{\rm gen}$. In particular, this intersection either has expected codimension or is empty. Furthermore, if the base field has characteristic $0$, there is a nonempty open subset $U_Y^{\rm sm}$ such that $v(G/P) \cap \cY^\circ$ is smooth.
\end{lemma}

\begin{proof} Define $Z = \{(v, x) \in U \times G/P \mid v(x) \in
  Y^\circ\}$. Let $\pi_1 \colon Z \to U$ and $\pi_2 \colon Z \to G/P$
  be the projection maps. We claim that $\pi_2$ is a fiber bundle with
  smooth fibers. Fix $x \in G/P$ and let $\cU(x)$ be the fiber of
  $\cU$ at $x$. Then the restriction map $\rho_x \colon U \to
  \cU(x)$ is $G'$-linear and it is surjective since $\cU(x)$ is
  irreducible as a $G'$-module. Hence the map $\rho_x \colon
  \rho_x^{-1}(Y^\circ) \to Y^\circ$ is an affine bundle. But $Y^\circ$
  is smooth, and $\rho_x^{-1}(Y^\circ) = \pi_2^{-1}(x)$, so the claim
  follows. In particular, $Z$ is smooth and
  \[
  \dim Z = \dim G/P + \dim U - \codim(Y^\circ, U').
  \]

  Now consider the map $\pi_1$. If it is dominant, then the fibers
  over a nonempty open subset $U_Y^{\rm gen}$ have dimension $\dim G/P
  - \codim(Y^\circ, U')$ and hence have expected
  codimension. Otherwise, the fibers are empty over a nonempty open
  subset $U_Y^{\rm gen}$. Given $v \in U$, we have an identification
  $\pi_1^{-1}(v) = v(G/P) \cap \cU$, so this proves the first claim.

  The last statement follows from generic smoothness applied to
  $\pi_1$.
\end{proof}


We will define $U^{\rm gen} \subset U$ to be the intersection of
$U_Y^{\rm gen}$ over all orbits $Y$ in $U'$, and similarly we define
$U^{\rm sm}$.

\begin{remark} The idea of studying degeneracy loci using perfect
  complexes rather than cohomology class formulas has been considered
  by the second author in \cite{sam} for the class of Schubert
  determinantal loci.
\end{remark}

\begin{example} \label{eg:ramanan}
Let $V$ be a vector space of dimension $2n$ and let $q \in S^2 V$ be a nondegenerate quadratic form. Denote the quotient $S^2 V / \langle q \rangle$ by $S^2_0 V$.

Consider the action of $\SO(V) \times \bC^*$ on $S^2_0 V$. Given a nondegenerate quadric $q' \in S^2_0 \bC^{2n}$, we can form the pencil $xq + yq'$. The determinant of the associated symmetric matrix gives us $2n$ points in $\bP^1$ and hence a hyperelliptic curve $C$ of genus $n-1$, and this process is reversible. This situation was considered by Weil.

Now consider the intersection of the quadrics defined by $q$ and $q'$ in $\bP(V^*)$. Then the variety of $\bP^{n-2}$'s in $q \cap q'$ is isomorphic to the Jacobian of $C$ (after fixing a base point), and the variety of $\bP^{n-3}$'s in $q \cap q'$ is isomorphic to the moduli space $SU_C(2,L)$ (see \S\ref{sec:vectorbundles}) where $L$ is any line bundle of odd degree. See \cite[Theorems 1, 2]{desale} for details.

These constructions can be interpreted as degeneracy loci as follows. Consider the orthogonal Grassmannian ${\bf OGr}(n-1,V^*)$, which is the subvariety of $\Gr(n-1,V^*)$ whose points are totally isotropic subspaces for $q$. The trivial bundle $\Gr(n-1,V^*) \times V^*$ has a tautological rank $n-1$ subbundle $\cR = \{(x,W) \mid x \in W\}$. Then we have $S_0^2 V = \rH^0({\bf OGr}(n-1,V^*); S^2 \cR^*)$ and $q'$ gives a generic section whose zero locus is the variety of $\bP^{n-2}$'s in $q \cap q'$. Similar comments apply to the variety of $\bP^{n-3}$'s using ${\bf OGr}(n-2,V^*)$. Modular interpretations for the degeneracy loci for the other Grassmannians ${\bf OGr}(k,V^*)$ are given by Ramanan \cite[\S 6, Theorem 3]{ramanan}.
\end{example}

\section{$\bC^5 \otimes \bigwedge^2 \bC^5$.}

In the rest of the article, we will work over the complex numbers
$\bC$. In fact, many results will hold over more general fields, but
as we have not done a systematic investigation of the correct
hypotheses, we will not make any attempt to be more general.

The analysis of the representation $\bC^5 \otimes \bigwedge^2 \bC^5$
is in fact easy to handle by more direct means, but we want to
illustrate our approach. We also mention that Fisher has examined this
case as well, see \cite{fisher1} and \cite{fisher2}.

Let $A$ and $B$ be vector spaces of dimension 5. The relevant data:

\begin{compactitem}
\item $U = A \otimes \bigwedge^2 B$
\item $G = (\GL(A) \times \GL(B)) / \{(x,x^{-1}) \mid x \in \bC^*\}$
\item $G/P = \bP(A^*) = \Gr(1,A^*)$
\item $\cU = \cR^* \otimes \bigwedge^2 \ul{B} \cong \cO(-1) \otimes
  \bigwedge^2 \ul{B}$
\item $U' = \bigwedge^2 \bC^5$
\item $G' = \GL_5(\bC)$
\end{compactitem}

The ring of invariants $\Sym(U^*)^{(G,G)}$ is a polynomial ring with generators of degrees 20, 30, and the graded Weyl group is Shephard--Todd group 16 \cite[\S 9]{vinberg}.

\subsection{Modules over $\cO_{U'}$.}

We are only interested in the ideal of $4 \times 4$ Pfaffians of
$U'$. This situation was explained in
Example~\ref{eg:buchsbaumeisenbud}. 

\subsection{Geometric data from a section.} \label{sec:geomc5w25}

The constructions in this section work over an arbitrary field.

We get the following locally free resolution over $\cO_{\cU} =
\Sym(\bigwedge^2 \ul{B}^* \otimes \cO_\cU(1))$:
\[
0 \to (\det \ul{B}^*)^{\otimes 2} \otimes \cO_{\cU}(-5) \to (\det
\ul{B}^*) \otimes \ul{B}^* \otimes \cO_{\cU}(-3) \to \bigwedge^4
\ul{B}^* \otimes \cO_{\cU}(-2) \to \cO_{\cU} \to \cO_{\cC} \to 0
\]
where $\cC$ has codimension 3 in the total space of $\cU$. Its
singular locus is the zero section of $\cU$, and has codimension 10 in
$\cU$.

For $v \in U^{\rm gen}$, $C_v = \cC \cap v(\bP(A^*))$ will have
codimension 3 in $v(\bP(A^*)) \cong \bP(A^*)$. By generic perfection,
we get a locally free resolution for $\cO_C$:
\begin{align*}
  0 \to (\det \ul{B}^*)^{\otimes 2} \otimes \cO_{\bP(A^*)}(-5) \to
  (\det \ul{B}^*) \otimes \ul{B}^* \otimes \cO_{\bP(A^*)}(-3)\\*
  \to \bigwedge^4 \ul{B}^* \otimes \cO_{\bP(A^*)}(-2) \to
  \cO_{\bP(A^*)} \to \cO_{C} \to 0.
\end{align*}
This gives enough information to see that $\omega_C = \cO_C$, $\dim
\rH^0(C; \cO_C) = 1$, and that $\dim \rH^1(C; \cO_C) = 1$. In
particular, $C$ is a curve of genus 1. We can also deduce that $C$ is
projectively normal and embedded by a complete linear series.

Conversely, given a smooth curve $C$ of genus 1 embedded in $\bP(A^*)$ by a complete linear series, its homogeneous ideal $I$ is generated by 5 quadrics and is a codimension 3 Gorenstein ideal. The Buchsbaum--Eisenbud classification of such ideals says that we can recover a section $v \in U$ which gives rise to $C$. 

\begin{theorem}
We have a bijection between $G$-orbits in $U^{\rm sm}$ and the set of pairs $(C,\cL)$ where $C$ is a genus $1$ curve and $\cL$ is a degree $5$ line bundle on $C$.
\end{theorem}

\subsection{Examples of singular quintic curves.}

In this section, we give a few examples of degenerations of the smooth elliptic quintic $C$. Describing degenerations of the Abelian varieties in the later examples will require more effort and will appear in future work.

We pick homogeneous coordinates $z_1, \dots, z_5$ on $\bP^4$.

\begin{example} Here is one example of a section that gives a rational
  nodal curve:
  \[
  \begin{pmatrix}0& {{z}_{5}}& {{z}_{1}}& {{z}_{2}}& {z}_{3}\\
    -{z}_{5}& 0& {z}_{2}& {z}_{3}&    {{z}_{4}}\\
    -{z}_{1}& {-{z}_{2}}& 0& {z}_{4}&    {{z}_{5}}\\
    -{z}_{2}& {-{z}_{3}}& {-{z}_{4}}& 0&    0\\
    {-{z}_{3}}& -z_{4}& {-z}_{5}& 0&    0
    \end{pmatrix}
  \]
  It is given by the parametrization $[a:b] \mapsto [a^5+b^5 : ab^4 :
  a^2b^3 : a^3b^2 : a^4b]$ and its node is the point
  $[1:0:0:0:0]$. The stabilizer subgroup in $G$ of this curve is the
  dihedral group of size 10 generated by the transformations $[a:b]
  \mapsto [b:a]$ and $[a:b] \mapsto [a:\zeta b]$ where $\zeta$ is a
  primitive 5th root of unity \cite[proof of Lemma 2.3]{fisher2}.
  This is  not an unstable orbit.

  Furthermore, its secant variety is an irreducible quintic
  hypersurface.
\end{example}

\begin{example} We can get a triangle consisting of two smooth
  rational quadrics and a line. Here is one example:
  \[
  \begin{pmatrix}0& 0& {z}_{4}& {z}_{3}&    {z}_{2}\\
    0& 0& 0& {z}_{2}&    {z}_{1}\\
    {-{z}_{4}}& 0& 0& 0&    {-{z}_{5}}\\
    {-{z}_{3}}& {-{z}_{2}}& 0& 0&    {-{z}_{4}}\\
    {-{z}_{2}}& {-{z}_{1}}& {z}_{5}& {z}_{4}& 0
  \end{pmatrix}
  \]
  The quadrics are $[a:b] \mapsto [a^2:ab:b^2:0:0]$ and $[a:b]\mapsto
  [0:0:a^2:ab:b^2]$ and the line is $[a:b] \mapsto [a:0:0:0:b]$. The
  secant variety is the union of $z_4 = 0$, $z_2 = 0$, and the cubic
  $z_1z_4^2 + z_2^2z_5 - z_1z_3z_5 = 0$.
\end{example}
  
\begin{example} 
  We can also get a union of 5 $\bP^1$'s which are labeled with $i \in
  \bZ/5$ such that $\bP_i^1$ intersects $\bP^1_j$ if and only if $j = i
  \pm 1$, and the intersection points are distinct. This is a {\bf
    N\'eron pentagon}. All N\'eron pentagons form a single orbit since
  they are determined by their points of intersection. Here is one
  example of a section that gives a N\'eron pentagon:
  \[
  \begin{pmatrix} 0& {z}_{1}& {z}_{2}& 0&    0\\
    {-{z}_{1}}& 0& 0& {z}_{3}&    0\\
    {-{z}_{2}}& 0& 0& 0&    {z}_{4}\\
    0& {-{z}_{3}}& 0& 0&    {z}_{5}\\
    0& 0& {-{z}_{4}}& {-{z}_{5}}& 0 \end{pmatrix}.
  \]
This is the set of points in $\bP(A^*)$ with at least 3 coordinates equal to 0. Its stabilizer subgroup contains the normalizer of the diagonal subgroup in $\SL(A)$. So the orbit of N\'eron pentagons has codimension at least 5.
  
Its secant variety is the hypersurface $z_1z_2z_3z_4z_5 = 0$. 
\end{example}

\begin{example} Here is a non-reduced example of a union of a rational
  normal cubic and a non-reduced line which intersect with
  multiplicity 2:
  \[
  \begin{pmatrix}0& 0& {z}_{5}& {z}_{3}&    {z}_{2}\\
    0& 0& 0& {z}_{2}&    {z}_{1}\\
    {-{z}_{5}}& 0& 0& {z}_{4}&    {z}_{3}\\
    {-{z}_{3}}& {-{z}_{2}}& {-{z}_{4}}& 0&    0\\
    {-{z}_{2}}& {-{z}_{1}}& {-{z}_{3}}& 0& 0
  \end{pmatrix}.
  \]
  The cubic is $[a:b]\mapsto [a^3:a^2b:ab^2:b^3:0]$ and the line is
  given by the ideal $(z_1,z_2,z_3^2)$.

  Its secant variety is the union of the hyperplane $z_5 = 0$ and the
  non-reduced quartic $(z_2^2 - z_1z_3)^2 = 0$.
\end{example}

\begin{example} Here is a rational cuspidal cubic:
\[
\begin{pmatrix}0& {z}_{1}& {z}_{4}& 0&  {z}_{5}\\
  {-{z}_{1}}& 0& 0& {z}_{5}&  {z}_{2}\\
  {-{z}_{4}}& 0& 0& {z}_{2}&  {z}_{3}\\
  0& {-{z}_{5}}& {-{z}_{2}}& 0&  {z}_{4}\\
  {-{z}_{5}}& {-{z}_{2}}& {-{z}_{3}}& {-{z}_{4}}& 0
      \end{pmatrix}.
\]
Its cusp point is $[0:0:1:0:0]$ and it has the parametrization
\[
[s:t] \mapsto [\sqrt{-1} t^5 : s^3t^2 : s^5 : \sqrt{-1} s^2t^3 : st^4
] .
\]
This vector lies in the unstable locus of the representation. The group of automorphisms of this curve that extend to automorphisms
of $\bP^4$ is generated by scaling $t$, so the orbit
of this curve in $\bC^5 \otimes \bigwedge^2 \bC^5$ has codimension 2. In particular, it gives a generic point of the unstable locus.

Its secant variety is an irreducible quintic hypersurface.
\end{example}

\subsection{Secant and tangential varieties.}

Here is a different approach:

\begin{compactitem}
\item $G/P = \Gr(2,A^*)$
\item $\cU = \cR^* \otimes \bigwedge^2 \ul{B}$
\item $U' = \bC^2 \otimes \bigwedge^2 \bC^5$
\item $G' = (\GL_2(\bC) \times \GL_5(\bC)) / \{(x,x^{-1}) \mid x \in \bC^*\}$
\end{compactitem}

The relevant $G'$-orbit closures in $U'$ are of codimensions 2, 4, and 5. The singular locus and the non-normal locus of the codimension 2 orbit closure are both the codimension 4 orbit closure $S'$. Also, $S'$ is smooth along the codimension 5 orbit closure $T'$. Furthermore, $S'$ and $T'$ can be identified as the secant and tangential varieties of the affine cone over the Segre variety $\bP^1 \times \Gr(2,5)$. Let $\cS$ and $\cT$ be the global versions of these varieties, and given a section $v \in U$, let $S$ and $T$ be the corresponding degeneracy loci.

\begin{proposition} $S$ is the locus of planes $W$ such that $\deg(\bP(W)
  \cap C) \ge 2$ and $T$ is the locus of planes $W$ such that $\bP(W)$
  is tangent to some point of $C$.
\end{proposition}

\begin{proof} Pick $W \in \Gr(2,A^*)$. Then $W \in S \setminus T$ if
  and only if $v(S) \in (\cR^* \otimes \bigwedge^2 \ul{B})(W)$ is a
  sum of the form $a_1 \otimes (b_1 \wedge c_1) + a_2 \otimes (b_2
  \wedge c_2)$ where $a_1$ and $a_2$ are linearly independent. But we
  can also identify this fiber with $\rH^0(\bP(W); \cO_{\bP(A^*)}(1)
  \otimes \bigwedge^2 \ul{B})$ when we identify $\bP(W)$ with a line
  in $\bP(A^*)$. This means that $\bP(W) \cap C$ consists of two
  points corresponding to the vanishing of $a_1$ and $a_2$. Since $T$
  is in the closure of $S \setminus T$, we finish via a limiting
  argument. 
\end{proof}

In particular, $T$ is a smooth curve. To calculate its genus, we can
use a free resolution for $T'$ (here we use $(\lambda; \mu)(-i)$ as
shorthand for $\bS_\lambda(\bC^2) \otimes \bS_\mu(\bC^5) \otimes
\Sym(U'^*)(-i)$):
\begin{align*}
  \bF_1 &= (2,1;2,1,1,1,1)(-3) + (2,2;2,2,2,2,0)(-4)\\
  \bF_2 &= (2,2;2,2,2,1,1)(-4) + (4,1;2,2,2,2,2)(-5) +
  (3,2;3,2,2,2,1)(-5) \\
  \bF_3 &= (4,2;3,3,2,2,2)(-6) + (3,3;4,2,2,2,2)(-6) +
  (4,3;3,3,3,3,2)(-7)\\
  \bF_4 &= (4,2;4,3,3,3,3)(-8) + (4,4;4,3,3,3,3)(-8)\\
  \bF_5 &= (6,4;4,4,4,4,4)(-10).
\end{align*}

So a locally free resolution for $T$ is
\begin{align*}
  \bF_1 &= \cR(-1)^5 + \cO(-2)^{15}\\
  \bF_2 &= \cO(-2)^5 + S^3(\cR)(-1) + \cR(-2)^{24}\\
  \bF_3 &= S^2(\cR)(-2)^{10} + \cO(-3)^{15} + \cR(-3)^5\\
  \bF_4 &= S^2(\cR)(-2)^5 + \cO(-4)^5\\
  \bF_5 &= S^2(\cR)(-4),
\end{align*}
and we see that $T$ has genus 1.

\begin{proposition} $C \cong T$
\end{proposition}

\begin{proof} Since $C$ is smooth, we have a well-defined morphism
  $\psi \colon C \to T$ obtained by sending $x \in C$ to the tangent
  line at $x$. This is a finite morphism, and by Riemann--Hurwitz, the
  ramification divisor is 0. Hence $\psi$ is an isomorphism. 
\end{proof}

\subsection{Chow forms.}

Yet another approach is as follows. 

\begin{compactitem}
\item $G/P = \Gr(3,A^*)$
\item $\cU = \cR^* \otimes \bigwedge^2 \ul{B}$
\item $U' = \bC^3 \otimes \bigwedge^2 \bC^5$
\item $G' = (\GL_3(\bC) \times \GL_5(\bC)) / \{(x,x^{-1}) \mid x \in \bC^*\}$
\end{compactitem}

The space $U'$ contains a $G'$-invariant degree 15 hypersurface. The
corresponding degeneracy locus $X'$ is a section of $\cO(5)$.

\begin{proposition} $X'$ is the Chow form of $C$.
\end{proposition}

\begin{proof}
  To obtain the Chow form of $C$, let $\Fl(1,3,A^*)$ be a partial flag
  variety with projections $\pi_1, \pi_2$ to $\bP(A^*)$ and
  $\Gr(3,A^*)$. Then $\rR {\pi_2}_* \rL \pi_1^* \cO_C$ is
  quasi-isomorphic to a complex whose determinant is the Chow form. A
  locally free resolution of $\cO_C$ in $\bP(A^*)$ is
  \[
  0 \to (\det B^*)^2(-5) \to \bS_{2,1^4} B^*(-3) \to \bigwedge^4
  B^*(-2) \to \cO_{\bP(A^*)}.
  \]
  Applying $\rR {\pi_2}_* \pi_1^*$ to this complex gives a 2-term
  complex over $\cO_{\Gr(3,A^*)}$:
  \begin{align} \label{eqn:C5alt2C5chow}
  0 \to (\det B^*)^2 \otimes S^2 \cR(-1) \to \cO_{\Gr(3,A^*)} \oplus
  \bS_{2,1^4} B^*(-1).
  \end{align}
  In this case, the determinant is just the actual determinant of a $6
  \times 6$ matrix. This gives a section of $\cO(5)$ which is the Chow
  form of $C$. 

  We claim that this map is a sheafy version of the following map. For
  $U' = {\bC^3} \otimes {\bigwedge^2 \bC^5}$, we have $(\det
  \bC^5)^{-2} \otimes \bS_{3,1,1} (\bC^3)^* \subset S^5(U'^*)$ with
  multiplicity 1, and also $(\det \bC^5)^{-2} \otimes \bS_{3,1,1}
  (\bC^3)^* \subset S^2(U'^*) \otimes (\det \bC^3)^* \otimes
  \bS_{2,1^4} (\bC^5)^*$ with multiplicity 1, so this gives a $6
  \times 6$ matrix
  \[
  (\det \bC^5)^{-2} \otimes \bS_{3,1,1} (\bC^3)^* \otimes \Sym(U'^*)
  \to \Sym(U'^*)(5) \oplus [(\det \bC^3)^* \otimes
  \bS_{2,1^4}(\bC^5)^* \otimes \Sym(U'^*)(2)]
  \]
  whose determinant is the degree 15 invariant. 

  Taking sections of \eqref{eqn:C5alt2C5chow}, we see that the process
  of constructing the Chow form of $C$ is a map that takes a section
  $v$ to an element in $[(\det B)^2 \otimes \bS_{3,1,1} A] \oplus
  [\wedge^4 B \otimes S^2 A]$. This can be interpreted as $\GL(A)
  \times \GL(B)$-equivariant linear maps $S^5(A \otimes \bigwedge^2 B)
  \to (\det B)^2 \otimes \bS_{3,1,1} A$ and $S^2(A \otimes \bigwedge^2
  B) \to \wedge^4 B \otimes S^2 A$. But such maps are unique up to
  scalar (checked with LiE \cite{lie}), so the claim follows.
\end{proof}

\subsection{Projective duality.}

Here is another approach. 

\begin{compactitem}
\item $G/P = \Gr(4,A^*) = \bP(A)$
\item $\cU = \cR^* \otimes \bigwedge^2 \ul{B}$
\item $U' = \bC^4 \otimes \bigwedge^2 \bC^5$
\item $G' = (\GL_4(\bC) \times \GL_5(\bC)) / \{(x,x^{-1}) \mid x \in \bC^*\}$
\end{compactitem}

The space $U'$ contains a $G'$-invariant degree 40 hypersurface. The
polynomial $f$ is described as follows: let $a_1, \dots, a_{40}$ be a
basis for the Lie algebra $\bC \oplus \mathfrak{sl}_4 \oplus
\mathfrak{sl}_5$. Then for $x \in \bC^4 \otimes \bigwedge^2 \bC^5$, we
have $f(x) = \det(A_1 x \cdots A_{40} x)$.

The corresponding degeneracy locus is a degree 10 hypersurface $C'$ in
$\bP(A)$. By the Katz--Kleiman formula \cite[\S 2.3.C]{gkz}, the
projective dual of an Abelian variety $X \subset \bP^N$ is a
hypersurface of degree $(\dim X + 1)(\deg X)$. In our case, for $X =
C$ from the last section, we get a hypersurface of degree 10.

\begin{proposition} For generic $v \in V$, the projective dual of $C$
  is $C'$.
\end{proposition}

\begin{proof} Let $f$ be the equation for the hyperdiscriminant of
  $\bC^4 \otimes \bigwedge^2 \bC^5$. An element $x \in \bC^4 \otimes
  \bigwedge^2 \bC^5$ can be thought of as a $5 \times 5$ skew-symmetric
  matrix of linear forms on $\bP^3$. Generically, the $4 \times 4$
  Pfaffians give a 0-scheme of degree 5, and $f(x) = 0$ if and only if
  this 0-scheme is non-reduced.

  So for generic $v \in V$, we first write $V = \rH^0(\bP(A); \cQ
  \otimes \bigwedge^2 B)$. Then as a $4 \times 10$ matrix (along the
  fibers), $v$ has full rank 4 over each point in $\bP(A)$. Now pick
  $H \in \bP(A)$. Then $H$ is a hyperplane in $\bP(A^*)$ and we have a
  canonical identification $\rH^0(H; \bigwedge^2 B(1)) = (\cQ \otimes
  \bigwedge^2 B)(H)$ (where the notation $(H)$ means ``fiber at
  $H$''). So $H \cap C$ is identified with the 5 points mentioned
  above, and hence $H \in C^\vee$ if and only if $f(v(H)) = 0$, which
  proves our claim.
\end{proof}

\section{$\bigwedge^3 \bC^9$.}

Let $V$ be a vector space of dimension 9. The relevant data:
\begin{compactitem}
\item $U = \bigwedge^3 V$
\item $G = \GL(V)$
\item $G/P = \bP(V^*) = \Gr(1,V^*)$
\item $\cU = \bigwedge^2 \cQ^* \otimes \cR^* \cong \bigwedge^2 \cQ^*
  \otimes \cO(1) \cong \Omega^2(3)$
\item $U' = \bigwedge^2 \bC^8$
\item $G' = \GL_8(\bC)$
\end{compactitem}

The ring of invariants $\Sym(U^*)^{(G,G)}$ is a polynomial ring with generators of degrees 12, 18, 24, 30, and the graded Weyl group $W$ is Shephard--Todd group 32 \cite[\S 9]{vinberg}. 

\begin{remark} \label{rmk:genus2}
The invariants for $W$ acting on its reflection representation $\fh$ were explicitly calculated by Maschke \cite{maschke}. It is known that the GIT quotient $U/\!\!/G \cong \fh/W$ contains an open set which is isomorphic to the moduli space of genus 2 curves $C$ with a marked Weierstrass point (i.e., ramification point for the hyperelliptic map). This was shown by Burkhardt \cite{burkhardt}. See also \cite[\S 4.3]{dolgachevlehavi} for a modern treatment and further discussion.
\end{remark}

\begin{remark} The orbits in $\bigwedge^3 \bC^9$ were classified in \cite{elashvili}, but the connection to geometric objects as treated here is not made.
\end{remark}

\subsection{Modules over $\cO_{U'}$.}

The orbits in $U'$ are given by the vanishing of Pfaffians of various
sizes. We are only interested in the vanishing locus of the $8 \times
8$ Pfaffian and the vanishing locus of the $6 \times 6$ Pfaffians. The
latter is described in Example~\ref{eg:jozefiakpragacz}. We denote
their global versions in $\cU$ by $\cY$ and $\cX$, respectively.

\subsection{Geometric data from a section.} \label{sec:w39geom}

The subvariety $\cY$ has the following resolutions over $\cO_\cU =
\Sym(\bigwedge^2 \cQ \otimes \cO(-1))$:
\begin{align} \label{eqn:w39coble}
0 \to (\det \cQ) \otimes \cO_\cU(-4) \to \cO_\cU \to \cO_{\cY} \to 0.
\end{align}
We can simplify this by noting that $\det \cQ = \cO(1)$. 

From Example~\ref{eg:jozefiakpragacz}, the subvariety $\cX$ has this resolution:
\begin{equation} \label{eqn:w39jp}
\begin{split}
  0 \to \cO_\cU(-9) \to \bigwedge^2 \cQ \otimes \cO_\cU(-7) \to
  \bS_{2,1^6} \cQ \otimes \cO_\cU(-7) \\* \to (S^2 \cQ \otimes
  \cO_\cU(-4)) \oplus ((S^2\cQ)^* \otimes \cO_\cU(-5)) \\*
  \to \bS_{2,1^6} \cQ \otimes \cO_\cU(-4) \to \bigwedge^6 \cQ \otimes
  \cO_\cU(-3) \to \cO_\cU \to \cO_{\cX} \to 0.
\end{split}
\end{equation}

So for $v \in U^{\rm gen}$, we have that $Y = \cY \cap v(\bP(V^*))$
and $X = \cX \cap v(\bP(V^*))$ will have codimensions 1 and 6 in
$\bP(V^*)$, respectively.

The self-duality of the resolution for $\cO_X$ shows that $\omega_X = \cE xt^6(\cO_X, \cO(-9)) = \cO_X$.

\begin{theorem} \label{thm:w39calc}
For $v \in U^{\rm gen}$, we have $\rh^i(X; \cO_X) = \binom{2}{i}$. In particular, for $v \in U^{\rm sm}$, $X$ is a torsor over an Abelian surface.
\end{theorem}

(We will only do this calculation once. For the remaining examples, we leave it to the reader.)

\begin{proof}
First, we replace the sheaf of algebras $\cO_\cU$ with the structure sheaf of $\bP(V^*)$ in \eqref{eqn:w39jp}. In the notation of \S\ref{section:bott}, we have $\bP(V^*) = \Fl(1;V^*)$ and $\bS_\lambda \cQ \otimes \cO_{\bP(V^*)}(d) = \cR(\mu)$ where $\mu = (d, -\lambda_8, -\lambda_7, \dots, -\lambda_1)$, and $\rho = (8,7,\dots,1,0)$. In particular, when we add $\rho$ to the sequence $\mu$ for any term in homological degrees $\{2,3,4,5\}$ of \eqref{eqn:w39jp}, all cohomology vanishes by Theorem~\ref{thm:glbott} because there will always be a repeating term. For example, $\bS_{2,1^6} \cQ \otimes \cO(-4)$ has vanishing cohomology because $\mu + \rho = (4,7,5,4,3,2,1,0,-2)$. 

Now consider the remaining terms. For $\bigwedge^6 \cQ \otimes \cO(-3)$, we have $\lambda = (1^6)$, so $\mu + \rho = (5,7,6,4,3,2,1,0,-1)$. We can sort this using 2 consecutive swaps, and subtracting $\rho$ again leaves us with a sequence of all $-1$. Hence Theorem~\ref{thm:glbott} says $\rH^2(\bP(V^*); \bigwedge^6 \cQ \otimes \cO(-3)) = \det V$. By similar considerations (or Serre duality), one can show that $\rh^6(\bigwedge^2 \cQ \otimes \cO(-7)) = 1$. Finally, we already know that $\rh^0(\cO_{\bP(V^*)}) = \rh^8(\cO(-9)) = 1$.

Now the result follows from a spectral sequence argument (or equivalently by splicing \eqref{eqn:w39jp} into short exact sequences). The last statement follows from Theorem~\ref{thm:cohomologyabelian}.
\end{proof}

From \eqref{eqn:w39coble}, we see that $Y$ is a cubic hypersurface. In fact, it is the Coble cubic of $X$, see \cite{beauville} and \cite{coble1} for more information on the Coble cubic.

\begin{proposition} The polarization on $X$ induced by $\cO_X(1)$ is
  indecomposable and of type $(3,3)$.
\end{proposition}

\begin{proof} We have $\rh^0(\cO_X(1)) = 9$, so $X$ is embedded via a
  complete linear series. So $K(\cO(1)) = \bZ/D \oplus \bZ/D$ where $D
  = (3,3)$ or $D = (1,9)$. Since $X$ is the singular locus of a cubic
  hypersurface, we conclude that it is the intersection of the
  quadrics (partial derivatives) that contain it. However, an Abelian
  surface in $\bP^8$ with a $(1,9)$-polarization cannot be the
  intersection of the quadrics containing it \cite[Remark 3.2]{gross},
  so we conclude that the polarization is $(3,3)$.

  Furthermore, $X$ is not the product of 2 elliptic curves as a
  polarized Abelian variety. To see this, first consider elliptic
  curves $E_1, E_2 \subset \bP^2$ embedded as cubics. Then the product
  polarization is given by the Segre embedding $E_1 \times E_2 \subset
  \bP^2 \times \bP^2 \subset \bP^8$. The quadratic equations vanishing
  on $\bP^2 \times \bP^2$ give $\bigwedge^2 \bC^3 \otimes \bigwedge^2
  \bC^3$ and the quotient of $S^2(\bC^3 \otimes \bC^3)$ by this space
  is $S^2 \bC^3 \otimes S^2 \bC^3$, none of which vanish on $E_1
  \times E_2$. Hence $E_1 \times E_2$ is not the intersection of
  quadrics in $\bP^8$.
\end{proof}

\begin{remark} It could be the case that $X$ is abstractly isomorphic to the product of two elliptic curves if we ignore the polarizations. See for example \cite[\S 12]{poonen}.
\end{remark}

\begin{remark}
Given a curve $C$ of genus 2, the moduli space $SU_C(3)$ admits a degree 2 map to $\bP^8$ which is branched along a degree 6 hypersurface. This hypersurface is projectively dual to the Coble cubic (see \cite{nguyen} and \cite{ortega} for two different proofs of this, together with some more discussion of the hypersurfaces).
\end{remark}

\subsection{Macaulay2 code.} \label{section:w39m2code}

Here we give some Macaulay2 code for generating examples of
sections. To do this, first consider the short exact sequence of
vector bundles over $\bP(V^*)$:
\[
0 \to \bigwedge^2 \cQ^* \otimes \cO(1) \to \bigwedge^2 \ul{V} \otimes
\cO(1) \to \cQ^* \otimes \cO(2) \to 0.
\]
Taking sections, we get the inclusion $\bigwedge^3 V \subset
\bigwedge^2 V \otimes V$. Since this map is $\GL(V)$-equivariant, it
must be the comultiplication map. Thus given $v \in \bigwedge^3 V$, we
can comultiply to get an element of $\bigwedge^2 V \otimes V$, which
we may think of as a skew-symmetric $9 \times 9$ matrix $\phi_v$ of
linear forms on $\bP(V^*)$.

To restrict back to skew-symmetric matrix on $\cQ^*$, we restrict our
attention to an affine open set. Pick homogeneous coordinates $z_1,
\dots, z_9$ on $\bP(V^*)$ and consider the open set given by $z_9 =
1$. Then $z_1, \dots, z_8$ give a trivialization for $\cQ$ over this
open set, so the relevant data is the upper-left $8 \times 8$
submatrix $\phi'_v$ of $\phi_v$. Now our degeneracy loci correspond to
the usual Pfaffian loci of this submatrix. Assuming that there are no
components contained in the hyperplane $z_9 = 0$, we get the
homogeneous ideals of these degeneracy loci by saturating with respect
to $z_9$.

The function {\tt basicMat} below takes as input {\tt s = (i,j,k)} and
computes the matrix $\phi'_v$ for $v = z_i \wedge z_j \wedge
z_k$. Then we take random coefficients and calculate the Pfaffian
ideals. 

\begin{verbatim}
P=101;
R = ZZ/P[z_1..z_9];
basicMat = s -> (
     ans := mutableMatrix(0*id_(R^8));
     ans_(s_0-1,s_1-1) = -z_(s_2);
     ans_(s_1-1,s_0-1) = z_(s_2);
     if not(member(9,s)) then (
          ans_(s_0-1,s_2-1) = z_(s_1);
          ans_(s_2-1,s_0-1) = -z_(s_1);
          ans_(s_1-1,s_2-1) = -z_(s_0);
          ans_(s_2-1,s_1-1) = z_(s_0);
          );
     matrix ans
     )

M=0;
for i in subsets(1..9,3) do M = M + random(ZZ/P) * basicMat(i);
I = saturate(pfaffians(8,M),ideal(z_9));
J = saturate(pfaffians(6,M),ideal(z_9));
K = saturate(pfaffians(4,M),ideal(z_9));
\end{verbatim}

Then generically, $I$ is the ideal of the cubic hypersurface, $J$ is
the homogeneous ideal of the Abelian surface, and $K$ is the unit
ideal.

\begin{example} \label{eg:c3c3c3}
Here is an example calculation we can do with this. Define {\tt M} by 
\begin{verbatim}
M=0;
for i from 1 to 3 do 
for j from 4 to 6 do 
for k from 7 to 9 do M = M + random(ZZ/P) * basicMat({i,j,k});
\end{verbatim}
This gives us a generic vector in $\bC^3 \otimes \bC^3 \otimes \bC^3
\subset \bigwedge^3 V$, where the subspace corresponds to the choice
of a decomposition $V = \bC^3 \oplus \bC^3 \oplus \bC^3$. Using the
command {\tt primaryDecomposition J} we see that $J$ now defines a
reducible variety of degree 12 with 2 components each contained in
$z_1=z_2=z_3=0$ and $z_4=z_5=z_6=0$. In fact, we know that $J$ is
supposed to have degree 18. There will be a third component of degree
6 inside $z_7=z_8=z_9=0$ which we do not see because of our choice of
open affine (one can check this by modifying the definition of {\tt
  basicMat} to use a different affine trivialization). In fact, if we
intersect any two of the three ideals in the primary decomposition of
$J$, we will get an ideal generated by 6 linear equations and a cubic,
which is exactly a plane cubic curve.
\end{example}

\section{$\bigwedge^4 \bC^8$.}

Let $V$ be a vector space of dimension 8. The relevant data:
\begin{compactitem}
\item $U = \bigwedge^4 V$
\item $G = \GL(V)$
\item $G/P = \bP(V^*) = \Gr(1,V^*)$
\item $\cU = \bigwedge^3 \cQ^* \otimes \cR^* \cong \bigwedge^3 \cQ^*
  \otimes \cO(1) \cong \Omega^3(4)$
\item $U' = \bigwedge^3 \bC^7$
\item $G' = \GL_7(\bC)$
\end{compactitem}

The ring of invariants $\Sym(U^*)^{(G,G)}$ is a polynomial ring with generators of degrees 2, 6, 8, 10, 12, 14, 18, and the graded Weyl group $W$ is the Weyl group of type ${\rm E}_7$ \cite[\S 9]{vinberg}.

\begin{remark} \label{rmk:genus3flex}
Letting $\fh$ be the 7-dimensional reflection representation of $W$, it is known that the GIT quotient $U/\!\!/G \cong \fh/W$ has an open subset isomorphic to the moduli space of smooth plane quartics (i.e., non-hyperelliptic genus 3 curves) with a marked flex point (i.e., a point with tangency of order $\ge 3$; there are 24 of them for a generic curve). See for example \cite[\S IX.7, Remark 7]{dolgachevortland} (but note that the 21 mentioned there should be 24) or \cite[Proposition 1.11]{looijenga}.
\end{remark}

\subsection{Modules over $\cO_{U'}$.}

We are interested in 3 orbits in $U'$. The first is a degree 7
hypersurface $Y'$. The singular locus of $Y'$ is an orbit closure $X'$
of codimension 4 in $U'$, and the singular locus of $X'$ is an orbit
closure $Z'$ of codimension 7 in $U'$. Let $W$ be a 7-dimensional
vector space with $U' = \bigwedge^3 W$. Set $A = \Sym(\bigwedge^3 W^*)
= \cO_{U'}$. The minimal free resolutions of these orbit closures are
given as follows.

The minimal free resolution of $\cO_{X'}$ is given by
\begin{equation}
  \begin{split} \label{eqn:w37codim4}
    0 \to (\det W^*)^5 \otimes W^*(-12) \to (\det W^*)^4 \otimes
    \bigwedge^2 W^* (-10) \to \quad \quad \\*
    (\det W^*)^2 \otimes \bS_{2,1^5}(W^*) (-7) \to (\det W^*)^2 \otimes
    \bigwedge^4 W^*(-6) \to A \to \cO_{X'} \to 0.
\end{split}
\end{equation}

The minimal free resolution for $\cO_{Z'}$ is given by
\begin{equation}
\begin{split} \label{eqn:w37codim7}
  0 \to \bS_{7^6,6} W^*(-16) \to \bS_{7,6^5,5} W^*(-14) \to
  \bS_{6^2, 5^4, 4} W^*(-12) \to \bS_{5^3, 4^3, 3} W^*(-10) \\*
  \to \bS_{4^4, 3^2, 2} W^*(-8) \to \bS_{3^5,2,1} W^*(-6) \to
  \bS_{2^6} W^*(-4) \to A \to \cO_{Z'} \to 0.
\end{split}
\end{equation}
We will also use the fact that the local ring of $\cO_{X'}$ at the
generic point of $\cO_{Z'}$ is a Cohen--Macaulay ring of type 3, i.e.,
its canonical module is minimally generated by 3 elements. This calculation was done by Federico Galetto \cite{galetto}.

We set $M$ to be a certain twist of the cokernel of the dual of the
last differential in \eqref{eqn:w37codim4}. In particular, it has a
presentation of the form
\begin{align} \label{eqn:w37canonical} (\det W^*)^2 \otimes
  \bigwedge^2 W (-4) \to \det W^* \otimes W (-2) \to M \to 0.
\end{align}
So up to a grading shift, $M$ is the canonical module of $\cO_{X'}$. 

\begin{lemma} \label{lemma:w37algstructure} There is a
  $G'$-equivariant $\cO_{X'}$-linear isomorphism $S^2 M \cong
  I_{Z',X'}$.
\end{lemma}

\begin{proof} 
  From \eqref{eqn:w37canonical}, we get the presentation
  \[
  (\det W^*)^3 \otimes W \otimes \bigwedge^2 W \otimes \cO_{X'}(-6)
  \to (\det W^*)^2 \otimes S^2 W \otimes \cO_{X'}(-4) \to S^2 M \to 0.
  \]
  Also, the presentation of $I_{Z'}$ in $\cO_{U'}$ is given by
  \[
  (\det W^*)^3 \otimes \bS_{2,1} W \otimes \cO_{U'}(-6) \to (\det
  W^*)^2 \otimes S^2 W \otimes \cO_{U'}(-4) \to I_{Z'} \to 0.
  \]
  To get a presentation over $\cO_{X'}$ we have to add in the
  relations $(\det W^*)^3 \otimes \bigwedge^3 W$ which come from the
  ideal generators of $I_{X'}$. Hence we see that the $S^2 M$ and
  $I_{Z',X'}$ have the same presentations. By equivariance, we get an
  isomorphism of these presentations up to a choice of a scalar, and
  this implies the desired isomorphism $S^2 M \cong I_{Z',X'}$.
\end{proof}

We define $\cY$, $\cX$, $\cZ$, and $\cM$ to be the global versions of
$Y'$, $X'$, $Z'$, and $M$.

\subsection{Geometric data from a section.}
\label{section:geometricw48}

Now choose a section $v \in U^{\rm gen}$ and set $Y = v(\bP(V^*)) \cap
\cY$, $X = v(\bP(V^*)) \cap \cX$, and $Z = v(\bP(V^*)) \cap \cZ$.

To get the locally free resolution of $X$ over
$\cO_{\bP^7}$, we replace $W$ with $\cQ^*$ and replace $(-i)$ with
$\cO(-i)$. After using $\det \cQ = \cO(1)$ to simplify, we get
\begin{align} \label{eqn:w48resolution}
0 \to \cQ \otimes \cO(-7) \to \bigwedge^2 \cQ \otimes \cO(-6) \to
\bS_{2,1^5}(\cQ) \otimes \cO(-5) \to \bigwedge^4 \cQ \otimes \cO(-4)
\to \cO_{\bP^7} \to \cO_{X} \to 0.
\end{align}

Furthermore, note that $\omega_X = \cM \otimes \cO_{v(\bP(V^*))}$. From Lemma~\ref{lemma:w37algstructure}, we get an $\cO_X$-linear map $\mu \colon S^2 \omega_X \to \cO_X$ and hence an $\cO_X$-algebra structure on $\cO_X \oplus \omega_X$. We define $\tilde{X} = \Spec_{\cO_X}(\cO_X \oplus \omega_X)$. As in Theorem~\ref{thm:w39calc}, we can check that $\omega_{\tilde{X}} = \cO_{\tilde{X}}$ and that $\rh^i(\tilde{X}; \cO_{\tilde{X}}) = \rh^i(X; \cO_X \oplus \omega_X) = \binom{3}{i}$.

Define $U^{\rm sm+}$ to be the subset of $U^{\rm sm}$ where the Cohen--Macaulay type of $X$ along $Z$ remains 3.

\begin{proposition}
If $v \in U^{\rm sm+}$, then $\tilde{X}$ is smooth. In particular, $\tilde{X}$ is a torsor over an Abelian $3$-fold, and $X$ is the Kummer variety of $\tilde{X}$. 
\end{proposition} 

\begin{proof} 
Let $\cI_Z$ denote the ideal sheaf of $Z$ in $\cO_X$. We check smoothness locally. Let $P$ be a nonsingular point of ${X'}$. Then $\cI_{Z,P} = \cO_{X,P}$ and $\omega_{{X},P} = \cO_{{X},P}$, and $\mu_P$ is just the multiplication map. Then $\cO_{\tilde{{X}}, P} \cong \cO_{{X},P}[t] / (t^2 - 1)$ as a ring, so the points over $P$ are nonsingular. Now let $P \in Z$ be a singular point. Set $R = \cO_{{X},P}$. Then $\fm = \cI_{Z,P}$ is the maximal ideal of $R$. Write $\omega_R = \omega_{{X},P}$ and $S = \cO_{\tilde{{X}}, P}$. Then $S = R \oplus \omega_R$ as an $R$-module, and $\fn = \fm \oplus \omega_R$ is the unique maximal ideal of $S$: $S/\fn \cong R/\fm$, and any element not in $\fn$ is $(r,0)$ for $r \in R \setminus \fm$, so is a unit. Since the multiplication map $S^2 \omega_R \to \fm$ is surjective, we have that $\fn^2 = \fm \oplus \fm\cdot \omega_R$. Hence $\fn / \fn^2 \cong \omega_R / \fm \cdot \omega_R$. The dimension of this space over $R / \fm$ is the Cohen--Macaulay type of $R$, which is 3, and implies that $S$ is a regular local ring. So we have shown that $\tilde{X}$ is smooth.

Using Theorem~\ref{thm:cohomologyabelian}, the above facts imply that $\tilde{X}$ is a torsor over an Abelian 3-fold. Temporarily choose a point $P \in \pi^{-1}(Z)$ to be the origin of $\tilde{X}$ so that it acquires a group structure. The map $\iota$ which swaps points in the same fiber of $\pi$ gives an involution on $\tilde{X}$. Then $\iota$ fixes $P$ and there is an induced linear map on the tangent space at $P$. Since $\iota$ has order 2, the induced linear map is the negation map. But this is the derivative of the inversion map on $\tilde{X}$, so we conclude that $\iota$ is the inversion map on $\tilde{X}$. So $X$ is the Kummer variety $\tilde{X} / \langle \iota \rangle$.
\end{proof}

\begin{proposition} $\cL = \pi^* \cO_X(1)$ defines an indecomposable
  $(2,2,2)$-polarization on $\tilde{X}$.
\end{proposition}

\begin{proof} Note that
\[
\rh^0(\tilde{X}; \cL) = \rh^0(X; \pi_*\cL) = \rh^0(X; (\cO_X \oplus
\omega_X) \otimes \cO_X(1)) = 8,
\]
so the map $\tilde{X} \to \bP(V^*)$ is given by a complete linear
series. Hence $\cL = \cL(a,b,c)$ where $abc = 8$. By \cite[\S 2,
Corollary 4]{mumford1}, each of $a,b,c$ must be even, so we have
$a=b=c=2$. Furthermore, $(\tilde{X},\cL)$ is indecomposable since by
\cite[Theorem 4.8.2]{birkenhake}, if $(\tilde{X},\cL)$ is a product of
$s$ Abelian varieties, then the degree of the map $\tilde{X} \to X
\subset \bP(V^*)$ is $2^s$.
\end{proof}

In particular, this Abelian 3-fold is the Jacobian of some genus 3
curve $C$.

\begin{proposition} $C$ is not hyperelliptic.
\end{proposition}

\begin{proof} 
Using the locally free resolution \eqref{eqn:w48resolution}, the map $\rH^0(\cO_{\bP(V^*)}(n)) \to \rH^0(\cO_X(n))$ is surjective for all $n \ge 0$. Hence $X$ is projectively normal, so $C$ is not hyperelliptic by \cite[\S 2.9.3]{khaled}.
\end{proof}

\begin{remark} \label{rmk:coblequartic}
The quartic hypersurface has an interpretation as the embedding of $SU_C(2)$ (see Section~\ref{sec:vectorbundles}) via its theta divisor \cite{narasimhan2}. This is also known as the Coble quartic, see \cite{beauville} and \cite{coble2} for more details.
\end{remark}

\subsection{Flag variety.}

Here is another approach which was pointed out to us by Jack Thorne.
\begin{compactitem}
\item $G/P = \Fl(1,7,V^*)$
\item $\cU = \cR_1^* \otimes \bigwedge^3(\cR_7/\cR_1)^*$
\item $U' = \bigwedge^3 \bC^6$
\item $G' = \GL_6(\bC)$.
\end{compactitem}

We are interested in the smallest orbit closure in $U'$, which is the
affine cone over $\Gr(3,6)$. This has codimension 10. In
characteristic 0, its minimal free resolution is as follows (the
coordinate ring is $\Sym(\bigwedge^3 \bC^6)$ and we only list the
partitions).
\begin{align*}
\bF_1 &= (2,1^4)\\
\bF_2 &= (2^4,1) + (3,2,1^4)\\
\bF_3 &= (3^2,2^2,1^2) + (3^5) + (5,2^5)\\
\bF_4 &= (4^3,2^3) + (4^2,3^3,1) + (5,3^3,2^2)\\
\bF_5 &= (5,4^2,3^2,2)^{\oplus 2}\\
\bF_6 &= (5^2,4^3,2) + (5^3,3^3) + (6,4^3,3^2)\\
\bF_7 &= (7,4^5) + (5^5,2) + (6^2,5^2,4^2)\\
\bF_8 &= (6^4,5,4) + (7,6,5^4)\\
\bF_9 &= (7,6^4,5)\\
\bF_{10} &= (7^6).
\end{align*}

To see the ranks, we include the graded Betti table (this follows Macaulay2 notation, so a term $d$ in row $i$ and column $j$ represents the free module $A(-i-j)^{\oplus d}$ in homological degree $j$).

\begin{verbatim}
        0  1   2   3   4    5   6   7   8  9 10
 total: 1 35 140 301 735 1080 735 301 140 35  1
     0: 1  .   .   .   .    .   .   .   .  .  .
     1: . 35 140 189   .    .   .   .   .  .  .
     2: .  .   . 112 735 1080 735 112   .  .  .
     3: .  .   .   .   .    .   . 189 140 35  .
     4: .  .   .   .   .    .   .   .   .  .  1
\end{verbatim}

\begin{remark}
The resolution changes in characteristic 2:
\begin{verbatim}
       0  1   2   3   4    5   6   7   8  9 10
total: 1 35 141 302 735 1080 735 302 141 35  1
    0: 1  .   .   .   .    .   .   .   .  .  .
    1: . 35 140 190   .    .   .   .   .  .  .
    2: .  .   1 112 735 1080 735 112   1  .  .
    3: .  .   .   .   .    .   . 190 140 35  .
    4: .  .   .   .   .    .   .   .   .  .  1
\end{verbatim}
and in characteristic 3:
\begin{verbatim}
       0  1   2   3   4    5   6   7   8  9 10
total: 1 35 140 321 756 1082 756 321 140 35  1
    0: 1  .   .   .   .    .   .   .   .  .  .
    1: . 35 140 189  20    1   .   .   .  .  .
    2: .  .   . 132 736 1080 736 132   .  .  .
    3: .  .   .   .   .    1  20 189 140 35  .
    4: .  .   .   .   .    .   .   .   .  .  1 
\end{verbatim} 
In all other characteristics, the Betti numbers agree with
characteristic 0. 
\end{remark}

The corresponding degeneracy locus $X_2$ is a torsor over an Abelian 3-fold: For each partition $\lambda$ in the resolution for $\Gr(3,6)$, we get the sheaf $\cR_1^{|\lambda|/3} \otimes \bS_{\lambda}(\cR_7/\cR_1)$. Also, the canonical sheaf of $\Fl(1,7,V^*)$ is $(\det\cR_7)^7 \otimes \cR_1^7$, which is the last term of the resolution. This shows that $\omega_{X_2} = \cO_{X_2}$. As in Theorem~\ref{thm:w39calc}, we can use Borel--Weil--Bott to get that $\rh^i(X_2; \cO_{X_2}) = \binom{3}{i}$.

\begin{proposition} \label{prop:flagabelian}
$X_2 \cong \tilde{X}$.
\end{proposition}

\begin{proof} 
Let $\pi \colon \Fl(1,7,V^*) \to \bP(V^*)$ be the projection map. We claim that $\pi(X_2) = X$ and that $\pi|_{X_2}$ is a finite morphism of degree 2. Note that $\pi^* \cQ = V^*/\cR_1$ and $\pi^* \cR = \cR_1$. For $x \in \bP(V^*)$, pick $(x \subset H) \in \pi^{-1}(x)$. We have a surjection
  \[
  (\bigwedge^3 \cQ^* \otimes \cR^*)(x) \to (\bigwedge^3(\cR_7/\cR_1)^*
  \otimes \cR_1^*)(x \subset H)
  \]
and so $(x \subset H) \in X_2$ if and only if $v(x) \in (V^*/x)^*$ is a pure tensor when mapped to $(H/x)^*$. By Remark~\ref{rmk:kummerflag}, there exists such an $H$ if and only if $x \in X$, so $\pi(X_2) = X$. Furthermore, there exists exactly 2 such $H$ if $x \in X \setminus Z$ and there exists exactly 1 such $H$ if $x \in Z$. So $\pi|_{X_2}$ is a finite morphism of degree 2. We conclude that $X_2 \cong \tilde{X}$.
\end{proof}

\subsection{Projective duality.}

We can instead work with the following data:
\begin{compactitem}
\item $G/P = \Gr(7,V^*) = \bP(V)$
\item $\cU = \bigwedge^4 \cR^*$
\item $U' = \bigwedge^4 \bC^7$
\item $G' = \GL_7(\bC)$
\end{compactitem}

The orbit classifications in $\bigwedge^4 \bC^7$ and $\bigwedge^3 \bC^7$ are the same, so we can proceed as in \S\ref{section:geometricw48}. 

Under the identification $\Gr(7,V^*) = \bP(V)$, the bundle
$\bigwedge^4 \cR^*$ becomes $\bigwedge^3 \cQ^* \otimes \cO(1)$. Hence,
given a section $v \in U^{\rm sm+}$, we get a quartic hypersurface
$Y^d$ and a Kummer 3-fold $X^d$ in $\bP(V)$.

\begin{proposition} 
  We have isomorphisms $X \cong X^d$ and $Y^d \cong Y$. Furthermore,
  $Y$ and $Y^d$ are projectively dual varieties.
\end{proposition}

\begin{proof}
  We first show that $X \cong X^d$. Consider the variety $X_2 \subset
  \Fl(1,7,V^*)$ constructed in the previous section from $v$. Let
  $\pi_2 \colon \Fl(1,7,V^*) \to \Gr(7,V^*)$ be the projection map. We
  claim that $\pi_2(X_2) = X^d$ and that $\pi_2$ is finite of degree
  2. This will prove the claim.

  Note that $\pi_2^* \cR = \cR_7$. For $H \in \Gr(7,V^*)$, pick $(x
  \subset H) \in \pi_2^{-1}(H)$. We have a surjection
  \[
  (\bigwedge^4 \cR^*)(H) \to (\bigwedge^3 (\cR_7 / \cR_1)^*\otimes
  \cR_1^*)(x \subset H)
  \]
  given by comultiplication. So $(x \subset H) \in X_2$ if and only if
  $v(x) \in H^*$ is a pure tensor when mapped to $(H/x)^*$. After
  picking a volume form in $\bigwedge^7 \bC^7$, we can identify
  $\bigwedge^4 {\bC^7}^*$ and $\bigwedge^3 {\bC^7}$, in which case the
  rest of the argument is similar to the proof of
  Proposition~\ref{prop:flagabelian}.

So $X$ and $X^d$ are embedded by dual linear series. By \cite[Theorem 3.1]{pauly}, the Coble quartics of $X$ and $X^d$ are isomorphic and projective dual to one another.
\end{proof}

\subsection{Doing calculations.}

We have written some Macaulay2 code for calculating the degeneracy loci in \S\ref{section:geometricw48}. However, it is messy so we do not include it here. We will just comment on how one can practically go about these calculations and explain the analogue of Example~\ref{eg:c3c3c3}.

First, we need to find a way to calculate the ideals of the appropriate low codimension orbits in $\bigwedge^3 \bC^7$. An explicit construction of the equation $f$ for the degree 7 hypersurface was given in \cite[Remark 4.4]{kimura}. Namely, he constructs two symmetric $7 \times 7$ matrices $\phi(x)$ and $\phi^*(x)$ whose entries are homogeneous polynomials of degrees 3 and 4, respectively, such that $\phi(x) \phi^*(x) = f(x) I_7$, i.e., the pair $(\phi(x), \phi^*(x))$ is a matrix factorization for $f$. Explicit summation formulas are given for the entries of these matrices.

From \eqref{eqn:w37codim4}, the Jacobian ideal of $f$ is reduced and gives the equations for the codimension 4 orbit. The ideal for the codimension 7 orbit is given by the entries of the matrix $\phi^*(x)$. One possibility to check this is to use that the representation generating this ideal is $S^2 \bC^7$ (up to a power of determinant) \eqref{eqn:w37codim7} and that this representation has multiplicity 1 in $S^4(\bigwedge^3 {\bC^7}^*)$. 

In practical computations, the product of the matrices $\phi(x)$ and $\phi^*(x)$ can be calculated in a few seconds. A possible way to proceed is to evaluate these against a chosen section $v$ (after choosing an open affine in $\bP^7$ to work over) and then to take the product in order to calculate the equation of the codimension 1 degeneracy locus. If the section is generic, then the Jacobian ideal of this equation will be the same as evaluating the Jacobian ideal of $f$ at $v$ (note that calculating $f$ and its Jacobian ideal before evaluating may take a long time).

\begin{example}
Suppose we choose a decomposition $V = \bC^2 \oplus \bC^2 \oplus \bC^2 \oplus \bC^2$, and we pick a generic vector in $\bC^2 \otimes \bC^2 \otimes \bC^2 \otimes \bC^2 \subset \bigwedge^4 V$. In this case, the codimension 4 degeneracy locus will have 4 irreducible components, all of degree 6. Taking any two of these components, the affine cone over their intersection is a bidegree $(2,2)$ hypersurface in $\bC[x,y] \otimes \bC[s,t]$ whose multi-graded Proj is a nonsingular curve of genus 1 (this agrees with the data obtained in \cite[\S 3.1.2]{weiho})
\end{example}

\section{$\bigwedge^4_0 \bC^8$.}

Let $V$ be a vector space of dimension 8 equipped with a symplectic
form $\omega \in \bigwedge^2 V^*$. The symplectic form gives an
injective multiplication map $\bigwedge^2 V^* \to \bigwedge^4 V^*$ and
we define $\bigwedge^4_0 V = \bigwedge^4_0 V^*$ to be the
cokernel. The relevant data:
\begin{compactitem}
\item $U = \bigwedge^4_0 V$
\item $G = \Sp(V)$
\item $G/P = \bP(V^*) = \Gr(1,V^*)$
\item $\cU = \bigwedge^3_0 (\cR^\perp / \cR) \otimes \cR^* \cong
  \bigwedge^3_0 (\cR^\perp / \cR) \otimes \cO(1)$
\item $U' = \bigwedge^3_0 \bC^6$
\item $G' = \Sp_6(\bC)$
\end{compactitem}

The ring of invariants $\Sym(U^*)^{(G,G)}$ is a polynomial ring with generators of degrees 2, 5, 6, 8, 9, 12, and the graded Weyl group $W$ is the Weyl group of type ${\rm E}_6$ \cite[\S 9]{vinberg}.

\begin{remark} \label{rmk:genus3hyperflex}
Letting $\fh$ be the 6-dimensional reflection representation of $W$, it is known that the GIT quotient $U/\!\!/G \cong \fh/W$ has an open subset isomorphic to the moduli space of smooth plane quartics (i.e., non-hyperelliptic genus 3 curves) with a marked hyperflex point (i.e., a point with tangency of order 4; the existence of such a point is a codimension 1 condition on the space of plane quartics). See \cite[Proposition 1.15]{looijenga}.
\end{remark}

\subsection{Modules over $\cO_{U'}$.}

Let $W$ be a 6-dimensional vector space equipped with a symplectic
form, so we can write $U' = \bigwedge^3_0 W$ and $G' = \Sp(W)$. Set $A
= \Sym(\bigwedge^3_0 W^*) = \cO_{U'}$. 

There is a degree 4 $G'$-invariant hypersurface in $U'$, which we
denote by $Y'$. The other orbits $X'$ and $Z'$ have codimension 4 and
7, respectively. The minimal free resolution of $\cO_{X'}$ is
\begin{align} \label{eqn:w036codim4}
0 \to W^* \otimes A(-7) \to \bigwedge^2_0 W^* \otimes A(-6) \to S^2
W^* \otimes A(-4) \to \bigwedge^3_0 W^* \otimes A(-3) \to A \to
\cO_{X'} \to 0,
\end{align}
which we can calculate by Macaulay2. Furthermore, the last matrix has
corank 3 when specialized to a nonzero point in the highest weight
orbit, so the localization of $\cO_{X'}$ at such a point is a
Cohen--Macaulay ring of type 3.

We set $M$ to be a certain twist of the cokernel of the dual of the
last differential of \eqref{eqn:w036codim4}. In particular, it has a
presentation of the form
\begin{align} \label{eqn:w036canonical}
\bigwedge^2_0 W(-2) \to W(-1) \to M \to 0.
\end{align}

The orbit closure $Z'$ is the affine cone over the Lagrangian
Grassmannian. The minimal free resolution for $\cO_{Z'}$ is
\begin{equation}
\begin{split}
  0 \to A(-10) \to S^2 W^* (-8) \to \bS_{[2,1]} W^*(-7) \to
  \bS_{[2,1,1]} W^*(-6) \\*
  \to \bS_{[2,1,1]} W^*(-4) \to \bS_{[2,1]} W^*(-3) \to S^2 W^*(-2)
  \to A \to \cO_{Z'} \to 0
\end{split}
\end{equation}
(recall that $\bS_{[\lambda]}$ was defined in \eqref{eqn:symplecticschur}).

\begin{lemma} \label{lemma:w036algstructure}
  There is a $G'$-equivariant $\cO_{X'}$-linear isomorphism $S^2 M
  \cong I_{Z',X'}$.
\end{lemma}

\begin{proof} The proof is the same as for
  Lemma~\ref{lemma:w37algstructure}. 
\end{proof}

We define $\cY$, $\cX$, $\cZ$, and $\cM$ to be the global versions of
$Y'$, $X'$, $Z'$, and $M$.

\subsection{Geometric data from a section.} \label{section:geometricw048}

Now choose a section $v \in U^{\rm gen}$ and set $Y = v(\bP(V^*)) \cap
\cY$, $X = v(\bP(V^*)) \cap \cX$ and $Z = v(\bP(V^*)) \cap \cZ$.

To get the locally free resolution of $X$ over $\cO_{\bP^7}$, we
replace $W$ with $\cR^\perp / \cR$ and replace $(-i)$ with
$\cO(-i)$. So the locally free resolution for $\cO_X$ is
\[
0 \to \cR^\perp/\cR(-7) \to \bigwedge^2_0(\cR^\perp/\cR)(-6) \to
S^2(\cR^\perp/\cR)(-4) \to \bigwedge^3_0(\cR^\perp/\cR)(-3) \to
\cO_{\bP^7} \to \cO_X \to 0,
\]
and the locally free resolution for $\omega_X$ is
\[
0 \to \cO_{\bP^7}(-8) \to \bigwedge^3_0(\cR^\perp/\cR)(-5) \to
S^2(\cR^\perp/\cR)(-4) \to \bigwedge^2_0(\cR^\perp/\cR)(-2) \to
(\cR^\perp/\cR)(-1) \to \omega_X \to 0,
\]
Using Borel--Weil--Bott as in Theorem~\ref{thm:w39calc}, we get $\rh^0(\cO_X) = 1$ and $\rh^2(\cO_X) = 3$ and all other cohomology vanishes. So by Serre duality, $\rh^i(\cO_X \oplus \omega_X) = \binom{3}{i}$.

From Lemma~\ref{lemma:w036algstructure}, we get a $\cO_X$-linear
multiplication map $\mu \colon S^2 \omega_X \to \cO_X$ and hence an
$\cO_X$-algebra structure on $\cO_X \oplus \omega_X$. We define
$\tilde{X} = \Spec_{\cO_X}(\cO_X \oplus \omega_X)$ and let $\pi \colon
\tilde{X} \to X$ be the projection map.

We define $U^{\rm sm+}$ to be the subset of $U^{\rm sm}$ where the
Cohen--Macaulay type of $X$ along $Z$ remains 3.

\begin{theorem} If $v \in U^{\rm sm+}$, then $\tilde{X}$ is
  smooth. Furthermore, $\tilde{X}$ is a torsor over an Abelian
  $3$-fold, $\cL = \pi^* \cO_X(1)$ is an indecomposable
  $(2,2,2)$-polarization on $\tilde{X}$, and $\tilde{X}$ is not the
  Jacobian of a hyperelliptic curve.
\end{theorem}

\begin{proof} The proofs are analogous to the ones in
  \S\ref{section:geometricw48}.
\end{proof}

\section{$\spin(16)$.} \label{sec:spin16}

Let $B$ be a vector space of dimension 16 equipped with a quadratic form $q \in S^2(B^*)$. Let $Z(q)$ be the zero locus of this quadratic form. We let $\Spin(B)$ be the simply-connected double cover of $\SO(B)$. One construction is to realize it as a multiplicative subgroup of a Clifford algebra, and we set ${\bf G}\Spin(B)$ to be the Clifford group generated by $\Spin(B)$ and the scalar matrices. There are two spin representations $\spin^\pm(B)$. The choice of spin representation will not affect our results, so we pick one and call it $\spin(B)$. Furthermore, $Z(q)$ possesses two ``spinor bundles'' which we call $\cS^+$ and $\cS^-$. These can be constructed as pushforwards of line bundles from the flag variety of $\Spin(B)$, or more direct geometric means (see \cite{ottaviani} for details). There is a perfect pairing
\[
\cS^+ \otimes \cS^- \to \cO(-1)
\]
\cite[Theorem 2.8]{ottaviani} and the sections of $\cS^+(1)$ and $\cS^-(1)$ give the two half-spin representations (which follows from their descriptions as pushforwards of line bundles).

The relevant data:
\begin{compactitem}
\item $U = \spin(B)$
\item $G = {\bf G}\Spin(B)$
\item $G/P = Z(q)$ (quadric hypersurface)
\item $\cU = \cS^+(1)$
\item $U' = \spin^+(14)$
\item $G' = {\bf G}\Spin_{14}(\bC)$
\end{compactitem}

The ring of invariants $\Sym(U^*)^{(G,G)}$ is a polynomial ring with generators of degrees 2, 8, 12, 14, 18, 20, 24, 30, and the graded Weyl group is the Weyl group of type ${\rm E}_8$ \cite[\S 9]{vinberg}.

\begin{remark} \label{rmk:genus4}
A non-hyperelliptic genus 4 curve lies on a unique quadric in its canonical embedding, and the locus of curves $C$ where this quadric is singular (i.e., $C$ has a vanishing $\theta$-characteristic) has codimension 1. This condition also implies that $C$ has a unique degree 3 map to $\bP^1$. If we further impose that this map has a point with non-simple ramification, the locus loses another dimension. Letting $\fh$ be the 8-dimensional reflection representation of $W$, the GIT quotient $U/\!\!/G \cong \fh/W$ has an open subset which should be isomorphic to this moduli space. We could not find any mention of this in the existing literature.

To get the interpretation for $\bC^7 / W({\rm E}_7)$ (Remark~\ref{rmk:genus3flex}), one first realizes $\bC$ as the smooth locus of a cuspidal cubic. Then $(P_1, \dots, P_7)$ becomes 7 points in the plane and blowing them up gives a del Pezzo surface. Its anticanonical divisor gives a map to $\bP^2$ branched along the plane quartic mentioned in Remark~\ref{rmk:genus3flex}.

When we do the same thing for 8 points, the corresponding del Pezzo surface is mapped to a quadric cone in $\bP^3$ via twice its anticanonical divisor, and it is branched along a genus 4 curve. However, 8 points in the plane generically do not lie on a cuspidal cubic; this condition on the points corresponds to the ramification condition on the genus 4 curve mentioned above. This was brought to our attention by Igor Dolgachev.
\end{remark}

The representation $\spin(14)$ has four orbits $Z(f)', Y', X', Z'$, of interest, which are of codimensions 1, 5, 10, and 14, respectively. All of them are Gorenstein and have rational singularities, but we have not yet  calculated the minimal free resolutions for $X'$ and $Z'$ due to certain extension problems. The hypersurface $Z(f)$ has degree 8.

The results for this example are incomplete, so we will just state what we expect to be true. For a generic section $v \in U$, the above four orbits give degeneracy loci $Z(f), Y, X, Z$. What should happen is that there is a curve $C$ of genus 4 such that $X$ is the Kummer variety of $\Jac(C)$ and $Z$ is its singular locus consisting of 256 points. Furthermore, we should have $Y = SU_C(2)$ (see \S\ref{sec:vectorbundles} for the definition of $SU_C(2)$). 

One can check that $Z(f)$ is a quartic hypersurface in $Z(q)$, and it should be an analogue of a Coble hypersurface. The curve $C$ should be a non-hyperelliptic curve with vanishing theta characteristic: every non-hyperelliptic curve $C$ can be written as a  complete intersection of a quadric and cubic in its canonical embedding, and having a vanishing theta characteristic means that this quadric is singular. For comparison, the analogues of Coble hypersurfaces for genus 4 curves without a vanishing theta characteristic were studied in \cite{oxburypauly}.

We remark that having the minimal free resolution of $X'$ (and of an additional auxiliary module $M'$) will allow one to prove that $X$ is the Kummer variety of an Abelian 4-fold, but there do not seem to be any cohomological characterizations of Jacobians amongst Abelian 4-folds.

\section{$\bC^4 \otimes \spin(10)$.}

We can write $\bC^4 \otimes \spin(10) = \spin(6) \otimes \spin(10)$, so this can be considered as a subcase of $\spin(16)$ in the previous section.

Let $A$ be a vector space of dimension 4 and $B$ be a vector space of
dimension 10 equipped with an orthogonal form $\omega \in S^2(B^*)$. The relevant data:

\begin{compactitem}
\item $U = A \otimes \spin(B)$
\item $G = (\GL(A) \times {\bf G}\Spin(B)) / \{(x, x^{-1}) \mid x \in
  \bC^*\}$ 
\item $G/P = \bP(A) = \Gr(1,A)$
\item $\cU = \cQ \otimes \spin(\ul{B})$
\item $U' = \bC^3 \otimes \spin(10)$
\item $G' = (\GL_3(\bC) \times \Spin_{10}(\bC)) / \{(x, x^{-1}) \mid x \in
  \bC^*\}$
\end{compactitem}

The ring of invariants $\Sym(U^*)^{(G,G)}$ is a polynomial ring with generators of degrees 8, 12, 20, 24, and the graded Weyl group is Shephard--Todd group 31 \cite[\S 9]{vinberg}.

\subsection{Modules over $\cO_{U'}$.}

Let $A' = \bC^3$ and let $G' = (\GL(A') \times \Spin(B)) /
\{(x,x^{-1}) \mid x \in \bC^*\}$. Let $R = \Sym((A' \otimes
\spin(B))^*)$.

There is a degree 12 $G'$-invariant hypersurface $X' \subset U'$,
whose equation is described in \cite[\S 3]{gyoja}. One can interpret
the construction there as taking the determinant of a certain $3
\times 3$ symmetric matrix whose entries are quartic forms. More
precisely, the representation $\bS_{2,2}(\spin(B))$ contains a
$\Spin(B)$-invariant, so we get polynomials $P$ which span the
representation $\bS_{2,2}(A'^*) \subset S^4((A' \otimes
\spin(B))^*)$. Then $\bS_{2,2}(A'^*) \cong S^2(A') \otimes (\det
A'^*)^{2}$, so we can interpret the space of $P$ as the space of
linear functions on symmetric $3 \times 3$ matrices of the form
\begin{align} \label{eqn:C3spin10matrix} \phi \colon A' \to A'^*
  \otimes (\det A')^{2},
\end{align}
and the equation for $X'$ is the determinant of this matrix. 

\begin{proposition}
  The $2 \times 2$ minors of $\phi$ define a radical ideal of
  codimension $3$, and the corresponding variety is the singular locus
  $Z'$ of $X'$.
\end{proposition}

\begin{proof}
The maximal codimension of the submaximal minors of a symmetric
matrix is 3. Since this ideal is $G'$-equivariant, it must cut out a
union of orbit closures. Since they do not vanish on $X'$ and there
are no orbit closures of codimension 2, this ideal has codimension 3,
and in particular defines a Cohen--Macaulay variety.

There are 2 orbits of codimension 3. We can check on representatives
that the $2 \times 2$ minors vanish on only one of these orbits, and
that the defined scheme is generically reduced on the other
orbit. This gives that the ideal of minors is radical of codimension
3. To check the statement about the singular locus, it is enough to
check the Jacobian matrix of the ideal of minors on orbit
representatives.
\end{proof}

As a corollary, the resolution of $\cO_{Z'}$ is given as follows (see
Example~\ref{eg:GJT}):
\begin{align*}
  0 \to A'^* \otimes (\det A'^*)^{5} \otimes R(-16) \to
  \bS_{2,1}(A'^*) \otimes (\det A'^*)^{3} \otimes R(-12) \to\\*
  S^2(A'^*) \otimes (\det A'^*)^{2} \otimes R(-8) \to R \to \cO_{Z'}
  \to 0.
\end{align*}

We set $M$ to be a certain twist of the cokernel of $\phi$, namely, we
have the presentation
\[
A' \otimes (\det A'^*)^3 \otimes R(-8) \to A'^* \otimes (\det A'^*)
\otimes R(-4) \to M \to 0.
\]

For $P \in X' \setminus Z'$, $\phi_P$ has corank 1, so $M_P \cong
\cO_{X',P}$. For $P$ in the open orbit of $Z'$, then $\phi_P$ has
corank 2, so $M_P$ is minimally generated by 2 elements.

\begin{lemma} \label{lemma:c3s10algstructure}
  There is a $G'$-equivariant $\cO_{X'}$-linear isomorphism $S^2 M
  \cong I_{Z',X'}$.
\end{lemma}

\begin{proof} The proof is similar to the proof of
  Lemma~\ref{lemma:w37algstructure}. 
\end{proof}

The other codimension 3 orbit closure $Z'_2$ can be described as maps
$\spin(B)^* \to A'$ whose kernel contains a nonzero pure spinor. This
variety fails to be normal, and the minimal free resolution of its
normalization is
\begin{align*}
  0 \to \det A' \otimes S^5 A' \otimes R(-8) \to \det A' \otimes S^3
  A' \otimes B \otimes R(-6) \\*
  \to \det A' \otimes S^2 A' \otimes \spin(B)^* \otimes R(-5) \to R
  \oplus \det A' \otimes \spin(B) \otimes R(-3) \to \tilde{\cO}_{Z'_2}
  \to 0.
\end{align*}

We define $\cX$, $\cZ$, $\cZ_2$, and $\cM$ to be the global versions
of $X'$, $Z'$, $Z'_2$ and $M$. 

\subsection{Geometric data from a section.}

Now choose a section $v \in U^{\rm gen}$ and set $X = v(\bP(A)) \cap
\cX$ and $Z = v(\bP(A)) \cap \cZ$.

\begin{lemma}
  $v(\bP(A)) \cap \cZ_2 = \emptyset$.
\end{lemma}

\begin{proof}
Since $\tilde{\cO}_{Z'_2}$ is a perfect module, we can specialize its resolution by replacing $A'$ by $\cQ$, so that we get a complex 
  \[
  0 \to S^5 \cQ^*(-1) \to S^3 \cQ^* \otimes \ul{B}(-1) \to S^2 \cQ^*
  \otimes \spin(\ul{B})^*(-1) \to \cO_{\bP^3} \oplus
  \spin(\ul{B})(-1).
  \]
  If $v(\bP(A)) \cap \cZ_2$ is nonempty, then it consists of finitely
  many points. But taking sections above, we get an exact
  complex. Hence the intersection is empty.
\end{proof}

From Lemma~\ref{lemma:c3s10algstructure}, we get a $\cO_X$-linear map
$\mu \colon S^2 \cM \to \cO_X$, which gives an $\cO_X$-algebra
structure on $\cO_X \oplus \cM$. We set $\cO_{\tilde{X}} =
\Spec_{\cO_X}(\cO_X \oplus \cM)$ and let $\pi \colon \tilde{X} \to X$
be the structure map.

Define $U^{\rm sm+}$ to be the subset of $U^{\rm sm}$ where $\cM$ is minimally generated by 2 elements along $Z$.

\begin{theorem} If $v \in U^{\rm sm+}$, then $\tilde{X}$ is an Abelian surface and $\cL = \pi^* \cO_X(1)$ is an indecomposable $(2,2)$-polarization.
\end{theorem}

\begin{proof}
The proof is similar to the proofs in \S\ref{section:geometricw48}.
\end{proof}

From \eqref{eqn:C3spin10matrix}, we also get a symmetric matrix
\begin{align} \label{eqn:quadcomplex}
\phi_v \colon \cQ \to \cQ^* \otimes (\det\cQ)^2.
\end{align}

\subsection{Quadratic complexes.}

Here is another approach:
\begin{compactitem}
\item $G/P = \Gr(2,A)$
\item $\cU = \cQ_2 \otimes \spin(\ul{B})$
\item $U' = \bC^2 \otimes \spin(10)$
\item $G' = (\GL_2(\bC) \times \Spin_{10}(\bC)) / \{(x, x^{-1}) \mid x \in
  \bC^*\}$
\end{compactitem}

There is a $G'$-invariant quartic hypersurface in $U'$. Given $v \in
U^{\rm sm}$, the corresponding degeneracy locus is a smooth section
$Q$ of $\cO(2)$.

\begin{proposition} $Q$ is the quadratic complex associated to the
  symmetric matrix \eqref{eqn:quadcomplex} constructed in the previous
  section.
\end{proposition}

\begin{proof} There is a $\GL(A)$-equivariant isomorphism
\[
k \colon \rH^0(\Gr(2,A); (\det\cQ_2)^2) \to \rH^0(\bP(A); S^2(\cQ^*) \otimes (\det \cQ)^2) 
\]
constructed in \cite[\S 8]{narasimhan1} which associates to a
quadratic complex (i.e., section in the domain above) to a symmetric
matrix $\cQ \to \cQ^* \otimes (\det \cQ)^2$ whose determinant is the equation of a Kummer surface. Since both sides are isomorphic to the irreducible
$\GL(A)$-representation $\bS_{2,2} A$, this map is uniquely determined
up to scalar. In the previous section and in the construction above,
we have constructed two $\GL(A)$-equivariant maps $S^4(A \otimes
\spin(B)) \to \bS_{2,2} A$. Since $\bS_{2,2} A$ appears in $S^4(A
\otimes \spin(B))$ with multiplicity 1, this map is also unique up to
scalar. Hence the diagram
\[
\xymatrix{ S^4(A \otimes \spin(B)) \ar[d] \ar[dr] \\
  \rH^0(\Gr(2,A); (\det\cQ_2)^2) \ar^-k[r] & \rH^0(\bP(A); S^2(\cQ^*) \otimes (\det \cQ)^2)
}
\]
commutes (up to possibly nonzero scalar ambiguity), which proves our
claim. 
\end{proof}

\subsection{Doing calculations.}

Again, we have written code in Macaulay2 for calculating the quadratic complex in affine trivializations in $\bP^3$, but it is too messy to include. We instead explain the concepts behind the calculation.

We will explicitly calculate the symmetric matrix $\cQ \to \cQ^* \otimes (\det \cQ)^2$ starting with $v \in A \otimes \spin(B)$. First, note that this symmetric matrix is a section of $S^2 \cQ^* \otimes (\det \cQ)^2 \subset S^2 A^* \otimes (\det \cQ)^2$, and the space of sections of the latter is $S^2 A^* \otimes \bS_{2,2,2} A = S^2 A^* \otimes S^2 A^* \otimes (\det A)^2$.
  
First, we have $\bS_{2,2} A^* \subset S^4(A^* \otimes \spin(B)^*)$. This is a 20-dimensional space of quartics which is described explicitly in \cite[\S 3]{gyoja} by the polynomials $P(x,y,z,w)$. More specifically, a basis is given by those $x,y,z,w$ such that $x \le y$, $x < z$, $y < w$, $z \le w$, and $1 \le x,y,z,w \le 4$, i.e., such that $T = \tableau[scY]{x,y | z,w}$ is a semistandard Young tableau. Abbreviate this polynomial by $P_T$. Let $Q_T \in \bS_{2,2} A$ be the dual basis vectors.
  
Then given $v$, we can produce the element $p_v = \sum_T P_T(v) Q_T \in \bS_{2,2} A = \bS_{2,2} A^* \otimes (\det A)^2$. The inclusion $\iota \colon \bS_{2,2} A^* \to S^2 A^* \otimes S^2 A^*$ is defined by
\[
\tableau[scY]{x,y|z,w} \mapsto 2(xy \otimes zw + zw \otimes xy) - (xz
\otimes yw + xw \otimes yz + yw \otimes xz + yz \otimes xw).
\]
So $\iota(p_v) \in S^2 A^* \otimes S^2 A^* \otimes (\det A)^2$. Identifying the latter $S^2 A^* \otimes (\det A)^2$ as sections of $(\det \cQ)^2$, this can be interpreted as a $4 \times 4$ symmetric matrix whose entries are quadrics on $\bP^3$. Let $z_1, z_2,z_3,z_4$ be homogeneous coordinates. Delete the last row and column of this matrix. This gives the symmetric matrix $\cQ \to \cQ^* \otimes (\det \cQ)^2$ over the affine open set $z_4 = 1$. Take the ideals defined by the $i \times i$ minors of this matrix for $i=3,2$ and saturate with respect to $z_4$ to get homogeneous ideals for the degeneracy loci.

\small \noindent Laurent Gruson, 
Universit\'e de Versailles Saint-Quentin-en-Yvelines, 
Versailles, France \\
{\tt laurent.gruson@math.uvsq.fr}

~

\small \noindent Steven V Sam, 
University of California, Berkeley, CA, USA\\
{\tt svs@math.berkeley.edu}, \url{http://math.berkeley.edu/~svs/}

~

\small \noindent Jerzy Weyman, 
Northeastern University, Boston, MA, USA \\
{\tt j.weyman@neu.edu}, \url{http://www.math.neu.edu/~weyman/}

\end{document}